\input ./style/arxiv-vmsta.cfg
\documentclass[numbers,compress,v1.0.1]{vmsta}

\usepackage{enumerate}

\volume{3}% Updated by VTEXPTS2LaTeX.exe, 03.01.2017 11:33
\issue{4}% Updated by VTEXPTS2LaTeX.exe, 03.01.2017 11:33
\pubyear{2016}
\firstpage{325}% Updated by VTEXPTS2LaTeX.exe, 03.01.2017 11:33
\lastpage{364}% Updated by VTEXPTS2LaTeX.exe, 03.01.2017 11:33
\doi{10.15559/16-VMSTA70}% Updated by VTEXPTS2LaTeX.exe, 22.12.2016
\allowdisplaybreaks

\DeclareMathOperator*{\argmin}{arg\,min}
\newcommand\independent{\protect\mathpalette{\protect\independenT}{\perp}}
\def\independenT#1#2{\mathrel{\rlap{$#1#2$}\mkern2mu{#1#2}}}

\newtheorem{theorem}{Theorem}

\newtheorem{proposition}{Proposition}
\theoremstyle{definition}
\newtheorem{example}{Example}
\newtheorem{remark}{Remark}
\newtheorem{definition}{Definition}

\def\<{\langle}
\def\>{\rangle}

\startlocaldefs

\allowdisplaybreaks
\endlocaldefs

\begin{document}
\begin{frontmatter}

\title{Description of the symmetric convex random closed sets as
zonotopes from their Feret diameters}

\author{\inits{S.}\fnm{Sa\"{\i}d}\snm{Rahmani}\corref{cor1}}\email{said.rahmani@emse.fr}
\cortext[cor1]{Corresponding author.}
\author{\inits{J.-C.}\fnm{Jean-Charles}\snm{Pinoli}}\email{pinoli@emse.fr}
\author{\inits{J.}\fnm{Johan}\snm{Debayle}}\email{debayle@emse.fr}

\address{\'{E}cole Nationale Sup\'erieure des Mines de Saint-Etienne, SPIN/LGF UMR CNRS~5307, 158 Cours Fauriel, 42023 Saint-Etienne, France}

\markboth{S. Rahmani et al.}{Description of a random symmetric convex set by a random zonotope}

\begin{abstract}
In this paper, the 2-D random closed sets (RACS) are studied by means
of the Feret diameter, also known as the caliper diameter. More
specifically, it is shown that a 2-D symmetric convex RACS can be
approximated as precisely as we want by some random zonotopes
(polytopes formed by the Minkowski sum of line segments) in terms of
the Hausdorff distance. Such an approximation is fully defined from the
Feret diameter of the 2-D convex RACS. Particularly, the moments of the
random vector representing the face lengths of the zonotope
approximation are related to the moments of the Feret diameter random
process of the RACS.
\end{abstract}


\received{24 October 2016}% Updated by VTEXPTS2LaTeX.exe, 22.12.2016
%10:23
%
\revised{7 December 2016}% Updated by VTEXPTS2LaTeX.exe, 22.12.2016
%10:23
%
\accepted{14 December 2016}% Updated by VTEXPTS2LaTeX.exe, 22.12.2016
%10:23
\publishedonline{3 January 2017}
\end{frontmatter}

\section{Introduction}

\subsection{Context and objectives}
\sloppy The geometrical characterization of granular media (grains,
pores, fibers, etc.) is an important issue in materials and process
sciences. Indeed, several granular media can be modeled as random sets
where the heterogeneity of the particles is studied with a
probabilistic approach \citep{theseGalerne,thesepeyrega}. In this
context, the random closed sets (RACSs) have been particularly studied
\citep
{torquato2002random,molch95,chiu2013stochastic,ballanisurfacepaircorel}
to get geometrical characteristics of such granular media. A RACS
denotes a random variable defined on a probability space $ (\varOmega
,\mathfrak{A},P)$ valued in $(\mathbb{F},\mathfrak{F})$, the family of
closed subsets of $\mathbb{R}^d$
provided with the $\sigma$-algebra $\mathfrak{F}:=\sigma\lbrace\lbrace
F\in\mathbb{F}\ \vert F\cap X\neq\emptyset\rbrace\,X\in\mathfrak
{K}\rbrace$, where $\mathfrak{K}$ denotes the class of compact subsets
on $\mathbb{R}^d$.
In a probabilistic point of view, the distribution of a convex RACS is
uniquely determined from the Choquet capacity functional \citep
{molchanov1994asymptotic,heinrich1999central}.
However, such a description is not suitable for explicitly determining
the geometrical shape of the RACS. An alternative way is to describe a
RACS by the probability distribution of real-valued geometrical
characteristics (area, perimeter, diameters, etc.).

\subsection{Original contribution}

The aim of this paper is to show how such characteristics can be used
to describe the geometrical shape of a convex random closed set in
$\mathbb{R}^2$. It has already been shown \cite{GSI} that the moments
of the Feret diameter of a~convex random closed set in $\mathbb{R}^2$
can be obtained by the area measures on morphological transforms of it.
A Feret diameter (also known as caliper diameter) is a measure of a set
size along a specified direction. It can be defined as the distance
between the two parallel planes restricting the set perpendicular to
that direction.

A set $X\in\mathbb{R}^2$ is said to be central symmetric or, more
simply, symmetric if it is equal to the set $\breve{X}:=-X$. Note that
the Feret diameter is not sensitive to such a central symmetrization
\cite{molchanovLivre}. Indeed, for a nonempty compact convex set
$X\subset\mathbb{R}^2$, its symmetrized set $\frac{1}{2}(X\oplus\breve
{X})$ (see \citep{molch95,hoffmann2007weak}) has the same Feret
diameter as $X$.
Consequently, the Feret diameter of a convex set $X$ is not enough to
fully reconstruct $X$ (but only its symmetrized set).
However, the Feret diameter is still useful to describe the shape of
convex sets for two reasons. Firstly, a convex set $X$ and its
symmetrized set $\frac{1}{2}(X\oplus\breve{X})$ share a lot of common
geometrical descriptors (perimeter, eccentricity, etc.). Secondly,
there are many applications in which symmetric convex particles are
considered. In this way, the reported work is focused on the symmetric
convex sets (i.e., $X= \frac{1}{2}(X\oplus\breve{X})$). By abusing the
notation, the conditions ``nonempty and compact'' will be often omitted
in this paper. In other words, without explicit mentioning of the
contrary, a \textit{convex set} will refer to a nonempty compact convex
set.

In this paper, we show that the Feret diameter of a random symmetric
convex set can be used to define some approximations of it as random
zonotopes. The polygonal approximation of a deterministic convex set
has already been studied several times \citep
{mcclure1975polygonal,glasauer1996asymptotic,bronstein2008approximation,campi1994approximation}.
However, in most cases, the approximation is made by using the support
function, which is not available in most of the geometric stochastic
models. Random polygons have already been studied several times \citep
{dafnis2009asymptotic,barany2010variance,miles1964random}. However,
they are defined in different ways and for other objectives, and they
are not characterized from their Feret diameters. In our point of view,
a zonotope (which is a Minkowski sum of line segments) is described by
its faces (direction and length) and can be characterized by its Feret
diameter. We will show that the Feret diameter of a symmetric convex
set evaluated on a finite number of directions $N>1$ can be used to
define some approximations of it as zonotope. Such zonotope
approximations will be generalized to the random symmetric convex sets.
Therefore, a random symmetric convex set will be approximated by a
random zonotope, and such approximations will be characterized from the
Feret diameter of the random symmetric convex set. The considered
random zonotope will be uniquely determined by the lengths of its
faces, and their directions will be assumed to be known. The
approximations considered are consistent as $N\rightarrow\infty$ with
respect to the Hausdorff distance.

This work is a preliminary study in order to describe the geometrical
characteristics of a population of convex particles in the context of
image analysis. Indeed, such images of population of convex particles
can be modeled by stochastic geometric models. In such a model, the
projection of a~particle represented by a random convex set.
Consequently, this work can be used to get information on such convex
particles. In addition, when the particles are supposed to be
symmetric, they have a symmetric 2-D projection that can be fully
characterized by the Feret diameter. Such a symmetric hypothesis is
suitable in several industrial applications in chemical engineering
(gas absorption, distillation, liquid--liquid extraction, petroleum
processes, crystallization, etc.).

An area of application is the gas--liquid reactions. Indeed in a such
process, the gas bubbles can be modeled as an ellipsoid the 2-D
projections of which are ellipses (see \citep
{zhang2012method,zafari2015segmentation,buwa2002dynamics}).
The main area of application is crystals manufacturing. Indeed, many
crystals are 3-D zonohedrons, and their 2-D projections are zonotopes.
For example, the crystals of oxalate ammonium \cite
{ICSIA,ahmad2011geometric}, the crystals of calcium oxalate dihydrate
\cite{zhang2002morphological}, or the (L)-glutamic acid \cite{presl28}.
In such applications, the considered approximations coincide with the
real data.

\subsection{Outline of the paper}
The paper is organized as follows.
The first part is devoted to the case of a~deterministic symmetric
convex set $X$. Some properties of the Feret diameter are first
recalled, and then for any integer $N>1$, an approximation $X^{(N)}_0$
of $X$ as a zonotope \cite{eppstein1996zonohedra} is described. It is
shown that this approximation is consistent as $N\rightarrow\infty$
with respect to the \xch{Hausdorff}{Hausdorfs} distance \cite{schneider2013convex}.
A~more accurate zonotope approximation $\tilde{X}^{(N)}_0$ of $X$ that
is invariant up to a rotation is also defined with the consistency also
satisfied. This approximation is particularly interesting to describe
the geometrical shape of $X$.

The second part is devoted to a characterization of the random
zonotopes. First, we explore some properties of the random process
associated with the Feret diameter. Then we study the random zonotopes,
define some their classes, and discuss their descriptions by their
faces. Finally, we study the characterization of some random zonotopes
from their Feret diameters random process.

In the last part, we study a random symmetric convex set $X$. We show
that it can still be described as precisely as we want by a random
zonotope $X^{(N)}_0$ and up to a rotation by a random zonotope
$X^{(N)}_\infty$ with respect to the Hausdorff distance. We give the
properties of these approximates and show that they can be
characterized from the Feret diameter random process of $X$. In
particular, the mean and autocovariance of the Feret diameter random
process of $X$ can be used to get the mean and the variances of the
random vectors composed by the face lengths of their zonotope approximations.

\section{Description of a symmetric convex set as a zonotope from its
Feret diameter}

The aim of this section is to discuss how a convex set $X$ can be
described as a zonotope. We will show that $X$ can always be
approximated as precisely as we want by zonotopes and how such
zonotopes can be characterized from the Feret diameter of~$X$. First,
we need to recall the definition of the Feret diameter and some its properties.

\subsection{Feret diameter and the support function}

\smallskip

\begin{definition}[Support function]
Let $X\subset\mathbb{R}^2$ be a convex set. The support function of
$X$ is defined as
\begin{displaymath}
f_X: \Bigg\vert %
\begin{array}{rcl}
\mathbb{R}^2 & \longrightarrow&\mathbb{R}\\
x & \longmapsto& \sup_{s\in X}\langle x,s\rangle=\max_{s\in X}\langle
x,s\rangle,\\
\end{array} %
\end{displaymath}
where $\langle\cdot,\cdot\rangle$ denotes the Euclidean dot product.
\end{definition}
The support function allows us to fully characterize a convex set.
Indeed, any positive homogeneous convex real-valued function on $\mathbb
{R}^2$ is the support function of a convex set \cite
{schneider2013convex}. In the following, we give some important
properties of the support function. The proofs are omitted since they
can be found in the literature \citep{gardner1995geometric,schneider2013convex}.
\begin{proposition}[Properties of the support function]
\label{prop:supportfunct}
Let $X\subset\mathbb{R}^2$ be a~convex set. Its support function
satisfies the following properties:
\begin{enumerate}
\item\label{it:positivementhomogene} \textit{Positive homogeneity}:
$\forall r\geq0,\; f_X(rx)=rf_X(x)$.
\item\label{it:sousadditive} \textit{Subadditivity}: $ f_X(x+y)\leq
f_X(x)+f_X(y)$.
\item\label{it:Minkadditive} $f_{X\oplus Y}= f_X+f_Y $, where $\oplus$
denotes the Minkowski addition.
\item\label{it:Invariance} If $s$ is a vectorial similarity and $b\in
\mathbb{R}^2$, then $f_{s(X)+b}(x)=f_X(s(x))+\langle x,b\rangle$.
\item\label{it:reconstruction} Reconstruction:
\begin{equation}
X=\bigcap_{x\in\mathbb{R}^2}\bigl\lbrace y\in
\mathbb{R}^2 \big\vert\langle y,x\rangle\leq f_X(x)\bigr
\rbrace. \label{eq:reconstruction}
\end{equation}
\item\label{it:positiviteContOrig} If, in addition, $0\in X$, then
$f_X\geq0$.
\item\label{it:hausdorff} $d_H(X,Y)=\Vert f_X-f_Y\Vert_\infty$, where
$d_H$ denotes the Hausdorff distance, and $\Vert\cdot\Vert_\infty$ is
the uniform norm on the unit sphere.
\end{enumerate}
\end{proposition}

Items~\ref{it:positivementhomogene} and \ref{it:sousadditive} relate
the convexity of the support function, and expression \eqref
{eq:reconstruction} allows the reconstruction of a convex set from its
support function. Note that the positive homogeneity of the support
function involves that it can be completely determined on the Euclidean
unit sphere. We adopt the following representation for the support
function of $X$:
\begin{displaymath}
h_X: \Bigg\vert %
\begin{array}{rcl}
\mathbb{R} & \longrightarrow&\mathbb{R}\\
\theta& \longmapsto& h_X(\theta)=f_X(^t(-\sin(\theta),\cos(\theta
))),\\
\end{array} %
\end{displaymath}
which is a continuous and $2\pi$-periodic function.

Note that the Feret diameter of a convex set $X$, denoted $H_X$, can be
expressed by the support function as
\begin{equation}
\forall\theta\in\mathbb{R},\quad  H_X(\theta)=h_X(\theta)+h_{\breve{X}}(\theta), \label{eqdef:ferretDiam}\vadjust{\eject}
\end{equation}
where $\breve{X}$ is the usual notation for the symmetric set $-X$. It
is easy to see that the Feret diameter of $X$ coincides with the
support function of $X\oplus\breve{X}$, where $\oplus$ denotes the
Minkowski sum. Therefore, the functional $H_X$ is sufficient to fully
characterize the symmetrized body $\frac{1}{2}(X\oplus\breve{X})$.
Note that if $X$ is already symmetric, then $H_X$ fully characterizes
$X$. We recall some important properties of the Feret diameter.
\begin{proposition}[Properties of the Feret diameter]
Let $X$ be a convex set. Then its Feret diameter $H_X$ satisfies the
following properties:
\begin{enumerate}
\item\label{it:add} For tow convex sets $X$ and $Y$, $H_{X\oplus
Y}=H_{X}+H_{Y}$.
\item\label{it:prod} $\forall r\in\mathbb{R},\; H_{rX}=\vert r\vert
H_{X} $.
\item\label{it:rot} If $R_\eta$ is a rotation and $b\in\mathbb{R}^2$,
then $\forall\theta\in\mathbb{R},\; H_{R_\eta(X)+b}(\theta
)=H_{X}(\theta+\eta)$.
\item\label{it:period}\textit{$\pi$- periodicity}: $\forall\theta\in
\mathbb{R},\; H_X(\theta+\pi)= H_X(\theta) $.
\item\label{it:inclusion} For two symmetric bodies $X$ and $Y$,
$H_X\leq H_Y \Leftrightarrow X\subseteq Y$
\item\label{it:lemme} For any $\theta,\beta\in[0,2\pi]$,
\begin{equation}
H_X(\theta+\beta)\leq H_X(\theta)+ 2\bigg\vert\sin\biggl(\frac{\beta}{2}\biggr)\bigg\vert H_X\biggl(\theta+\frac{\beta+\pi}{2}\biggr). \label{Cond:convexPolygon}
\end{equation}
\end{enumerate}
\label{prop:feretprop}
\end{proposition}

\begin{proof}
$\:$
\begin{enumerate}[1,~2,~3.]
\item[1,~2,~3.] According to Eq.~\eqref{eqdef:ferretDiam}, the first
three items come directly from Proposition~\autoref{prop:supportfunct}.
\item[4.] The $\pi$-periodicity follows from $h_{\breve{X}}(\theta)=
h_{X}(\theta+\pi)$, $\theta\in\mathbb{R}$.
\item[5.] Because of the symmetry of $X$ and $Y$, if $H_X\leq H_Y$,
then $ h_X\leq h_Y$. Therefore, for any $x\in\mathbb{R}^2$, $f_X(x)\leq
f_Y(x)$, so $\lbrace y\in\mathbb{R}^2 \vert\<y,x\>\leq f_X(x)\rbrace
\subseteq\lbrace y\in\mathbb{R}^2 \vert\<y,x\>\leq f_Y(x)\rbrace$,
and thus $X\subset Y$ by Proposition~\ref{prop:supportfunct}.\ref
{it:reconstruction}.

Suppose that $X\subset Y$. Then $\forall x\in\mathbb{R}^2,\; \lbrace\<
s,x\>\vert s\in X\rbrace\subset\lbrace\<s,x\>\vert s\in Y\rbrace
\Rightarrow f_X(x)\leq f_Y(x)\Rightarrow h_X\leq h_Y\Rightarrow H_X\leq H_Y$.
\item[6.] For any $(\theta,\beta)\in\mathbb{R}^2$, let $\alpha=\beta+\pi
$, $x=^t(-\sin(\theta),\cos(\theta))$, $z=^t(-\sin(\theta+\alpha),\cos
(\theta+\alpha))$ and $y=z+x$, so that
\begin{align*}
&f_X(y-x)\leq f_X(-x)+f_X(y),
\\
& h_X(\theta+\alpha)\leq h_X(\theta+
\pi)+f_X(y),
\end{align*}
and
\begin{align*}
{\parallel} y{\parallel}&=\sqrt{\bigl(\sin(\theta)+\sin(\theta+\alpha)\bigr)^2+ \bigl(\cos (\theta)+\cos(\theta+\alpha)\bigr)^2}\\
&=\sqrt{(2+2\bigl(\sin(\theta)\sin(\theta+\alpha)+ \cos(\theta)\cos(\theta + \alpha)\bigr)}\\
&=\sqrt{2}\sqrt{1+\cos(\alpha)}\\
&=\sqrt{2}\sqrt{2\cos^2\biggl(\frac{\alpha}{2}\biggr)}\\
&=2\biggl\vert\cos\biggl(\frac{\alpha}{2}\biggr)\biggr\vert\\
&=2\biggl\vert\sin\biggl(\frac{\beta}{2}\biggr)\biggr\vert.
\end{align*}
Using the formulas
\begin{align*}
\sin(\theta)+\sin(\theta+\alpha) &=2\sin\biggl(\theta+\frac{\alpha}{2}\biggr)
\cos \biggl(\frac{\alpha}{2}\biggr),
\\
\cos(\theta)+\cos(\theta+\alpha) &=2\cos\biggl(\theta+\frac{\alpha}{2}\biggr)
\cos \biggl(\frac{\alpha}{2}\biggr)
\end{align*}
and taking $\eta\in\mathbb{R}$ such that $y=\parallel y\parallel
^t(-\sin(\eta),\cos(\eta))$, we have
\begin{align*}
\sin(\eta) &=\frac{2\sin(\theta+\frac{\alpha}{2})\cos(\frac{\alpha
}{2})}{\parallel y\parallel},
\\
\cos(\eta) &=\frac{2\cos(\theta+\frac{\alpha}{2})\cos(\frac{\alpha
}{2})}{\parallel y\parallel}.
\end{align*}
Let $s$ be the sign of $\cos(\frac{\alpha}{2})$. Then $\sin(\eta) =s\sin
(\theta+\frac{\alpha}{2})$ and $\cos(\eta) =s\cos(\theta+\frac{\alpha
}{2})$.

Finally, $\eta\in\lbrace\theta+\frac{\beta+\pi}{2},\theta+\frac{\beta
+\pi}{2}+\pi\rbrace$, and it can be expressed as
\begin{equation*}
h_X(\theta+\beta-\pi)\leq h_X(\theta+\pi)+2\biggl\vert\sin
\biggl(\frac{\beta
}{2}\biggr)\biggr\vert h_X(\eta).
\end{equation*}
This result is true for any convex set $X$, in particular, for $Y=\frac
{1}{2}( X\oplus\breve{X})$. However, $h_y=H_X$, and then by the $\pi
$-periodicity of the Feret diameter we have
\begin{equation*}
\forall\theta \quad \beta\in[0,2\pi], \quad H_X(\theta+\beta)\leq
H_X(\theta)+ 2\biggl\vert\sin\biggl(\frac{\beta}{2}\biggr)\biggr\vert
H_X\biggl(\theta+\frac{\beta+\pi}{2} \biggr).\qedhere
\end{equation*}
\end{enumerate}
\end{proof}

The Feret diameter can also be related to the mixed area \cite
{schneider2013convex} by using a line segment as a structural element.
Indeed, using the Steiner formula \cite{schneider2013convex} with two
convex sets $X$ and $Y$, we have
\begin{equation*}
A(X\oplus Y)=A(X)+2W(X,Y)+A(Y),
\end{equation*}
where $W(X,Y)$ denotes the mixed area between $X$ and $Y$. The mixed
area functional $W(\cdot,\cdot)$ is a symmetric mapping, which is also
homogeneous in its two variables (see \citep
{Kminch,schneider2013convex} for details). It is often used to describe
some morphological characteristics of a convex set $X$ by using
different structuring elements. For instance, if $X$ is a bounded
convex set and $B$ is the unit disk, then $W(X,B)=\frac{1}{2}U(X)$,
where $U(X)$ denotes the perimeter of $X$. Let $X$ be a bounded convex
set, and $S_\theta$ be a unit line segment directed by $\theta$. Then
\begin{equation}
W(X,S_\theta)=\frac{1}{2}H_X(\theta). \label{eq:MixedAreatFerteseg}
\end{equation}

The proof is omitted since it consists in a simple drawing and can be
found in the literature \citep{GSI,Kminch}.

\begin{remark}
This relation is very important because it involves an interpretation
of the mixed area of a convex set with the Minkowski addition of line
segments from its Feret diameter. Indeed, for any $\theta_1,\theta_2\in
[0,\pi]$ and $\alpha_1,\alpha_2\in\mathbb{R}_+$,
\begin{align*}
A(X\oplus\alpha_1 S_{\theta_1}\oplus\alpha_2
S_{\theta_2})&=A(X\oplus \alpha_1 S_{\theta_1})+2W(X\oplus
\alpha_1 S_{\theta_1},\alpha_2 S_{\theta_2})
\\
&=A(X)+\alpha_1 H_X(\theta_1)+
\alpha_2 H_{X\oplus\alpha_1 S_{\theta
_1}}(\theta_2)
\\
&=A(X)+\alpha_1 H_X(\theta_1)+
\alpha_2 H_{X}(\theta_2)+\alpha_2
H_{\alpha_1 S_{\theta_1}}(\theta_2).
\end{align*}
However, $ \alpha_2 H_{\alpha_1S_{\theta_1}}(\theta_2)=W(\alpha_1
S_{\theta_1},S_{\theta_2})=A(\alpha_1 S_{\theta_1}\oplus\alpha_2
S_{\theta_2})$. Then,
\begin{align*}
W(X,\alpha_1 S_{\theta_1}\oplus\alpha_2
S_{\theta_2})=\frac
{1}{2}\bigl(\alpha_1 H_X(
\theta_1)+\alpha_2 H_{X}(\theta_2)
\bigr).
\end{align*}
This result can be easily generalized by induction to any Minkowski sum
of line segments: $\forall n\geq1,\forall i=1,\dots, n,\; \alpha_i\in
\mathbb{R}_+,\theta_i\in\mathbb{R} $, we have
\begin{equation}
W\Biggl(X ,\bigoplus_{i=1}^n
\alpha_i S_{\theta_i}\Biggr)=\frac{1}{2}\sum
_{i=1}^n \alpha_i H_X(
\theta_i). \label{eq:linearWtoH}
\end{equation}
Relation \eqref{eq:linearWtoH} has an important kind of linearity.
Indeed, it implies formulae for the computation of the mixed area
between a convex set and a symmetric body from their Feret diameter
\textup{(}see Remark \autoref{rq:mixedareaapprox}\textup{)}.
\end{remark}

\subsection{Approximation of a symmetric convex set by a $0$-regular zonotope}

Now we give some properties of the Feret diameter of a convex set and
its connection with the mixed area. Here the zonotope will be defined
and particularly the class of the $0$-regular zonotopes, some
properties of the zonotopes will be discussed. In particular, we will
show how a symmetric convex set can be approximated by a $0$-regular
zonotope as precisely as we want.

Let $\mathcal{C}$ denote the class of all symmetric convex sets of
$\mathbb{R}^2$, where the symmetry is given in the sense of Minkowski:
$X=\frac{1}{2}(X\oplus\breve{X})$.
Let $S_0$ be the unit line segment $[-\frac{1}{2},\frac{1}{2}]$, and
$S_t$ its rotation by the angle $t\in[0,\pi[$. Consider now the convex
set $X$ such that
\begin{equation}
X=\bigoplus_{i=1}^n \alpha_i
S_{\theta_i}, \quad n\in\mathbb{N}^*, \ \forall i=1,\ldots n, \
\alpha_i\in\mathbb{R}_+, \ \theta_i\in[0,\pi[.
\label{eq:representation}
\end{equation}
Note that $X$ is a compact convex symmetric polygon with at most $2n$
faces, where for all $i=1,\dots,n$, $\alpha_i$ is the length of the
two faces of $X$ oriented by~$\theta_i$. It is easy to see that every
compact convex symmetric polygon has an even number of faces and can be
represented as \eqref{eq:representation} up to a translation.
Furthermore, note that $X$ has a nonempty interior if and only if $n>1$.

\begin{definition}[Zonotopes]
Any compact convex symmetric polygon such as \eqref{eq:representation}
is called a \textit{zonotope}. For $N\in\mathbb{N}^*$, $\mathcal
{C}^{(N)}$ denotes the set of all zonotopes with at most $2N$ faces:
\begin{equation*}
\mathcal{C}^{(N)}=\Biggl\lbrace\bigoplus_{i=1}^N
\alpha_i S_{\theta_i}\big\vert \alpha\in\mathbb{R}_+ ^N,
\; \theta\in[0,\pi[^N\Biggr\rbrace,
\end{equation*}
where $\alpha=^t(\alpha_1,\ldots\alpha_N)$ and $\theta=^t(\theta
_1,\ldots\theta_N)$.
\end{definition}
Several geometric characteristics and properties of zonotopes can be
easily expressed from representation \eqref{eq:representation}.

\begin{proposition}[Geometrical characterization of zonotopes]
Let $N\in\mathbb{N}^*$, and $X=\bigoplus_{i=1}^N \alpha_i S_{\theta_i}
$ be an element of $\mathcal{C}^{(N)}$. Let $H_X$ be its Feret diameter
function, $U(X)$ its perimeter, and $A(X)$ its area. Then
\begin{equation}
\forall\eta\in\mathbb{R},\; H_X(\eta)=\sum
_{i=1}^N \alpha_i\bigl\vert\sin (\eta-
\theta_i)\bigr\vert, \label{eq:ferretPolygone}
\end{equation}
\begin{equation}
U(X)=2\sum_{i=1}^N \alpha_i,
\label{eq:PerimPolygone}
\end{equation}
\begin{equation}
A(X)=\frac{1}{2}\sum_{i=1}^N \sum
_{j=1}^N\alpha_i
\alpha_j\bigl\vert\sin (\theta_i-\theta_j)\bigr\vert.
\label{eq:airPolygone}
\end{equation}

\label{prop:A P H deterministe}
\end{proposition}

\begin{proof}
$\:$
\begin{enumerate}
\item[(6)] For any $(\beta,\eta)\in\mathbb{R}^2$, the support function
of the line segment $S_\beta$ in the direction $\eta$ is
\begin{align*}
h_{S_\beta}(\eta)&= \max_{t\in[-\frac{1}{2},\frac{1}{2}]}\bigl\lbrace t\bigl(-\cos (\beta)\sin(\eta)+\sin(\beta)\cos(\eta)\bigr)\bigr\rbrace\\
& =\max_{t\in[-\frac{1}{2},\frac{1}{2}]}\bigl\lbrace t\sin(\beta-\eta )\bigr\rbrace\\
&= \frac{1}{2}\bigl\vert\sin(\beta-\eta)\bigr\vert\\
\Rightarrow  \quad  H_{S_\beta}(\eta)&=\bigl\vert\sin(\beta-\eta)\bigr\vert.
\end{align*}
Then relation \eqref{eq:ferretPolygone} follows from Propositions~\ref
{prop:feretprop}.\ref{it:add} and \ref{prop:feretprop}.\ref{it:prod}.
\item[(7)] If $X$ is a polygon of $2N$ faces of length $\alpha_i,
i=1,\dots, N$, the perimeter can be obtained by adding up the face lengths.
\item[(8)] For the area, the result \eqref{eq:airPolygone} is proved
by induction on $N$: for $N=1$, $X=S_{\theta_1}$ and $A(X)=0$, so that
\eqref{eq:airPolygone} is satisfied. Suppose that \eqref
{eq:airPolygone} is true for $n\leq N$ and let us show that it is true
for $N+1$. Since $ X=(\bigoplus_{i=1}^{N} \alpha_i S_{\theta_i})\oplus
\alpha_{N+1}S_{\theta_{N+1}} $, by the Steiner formula we have
\begin{align*}
A(X) &=A\Biggl(\bigoplus_{i=1}^{N}
\alpha_i S_{\theta_i}\Biggr) +2W\Biggl(\bigoplus
_{i=1}^{N} \alpha_i S_{\theta_i},
\alpha_{N+1}S_{\theta_{N+1}}\Biggr).
\end{align*}
Then, by \eqref{eq:MixedAreatFerteseg},
\begin{align*}
2W\Biggl(\bigoplus_{i=1}^{N}
\alpha_i S_{\theta_i}, \alpha_{N+1}S_{\theta
_{N+1}}
\Biggr)=\alpha_{N+1}H_{\bigoplus_{i=1}^{N} \alpha_i S_{\theta
_i}}(\theta_{N+1}),
\end{align*}
and finally, by the heredity assumption and \eqref{eq:ferretPolygone},
\begin{align*}
A(X) &=\frac{1}{2}\sum_{i=1}^N \sum_{j=1}^N\alpha_i\alpha_j\bigl\vert\sin (\theta_i-\theta_j)\bigr\vert+\alpha_{N+1}H_{\bigoplus_{i=1}^{N} \alpha_i S_{\theta_i}}(\theta_{N+1})\\
&=\frac{1}{2}\sum_{i=1}^N \sum_{j=1}^N\alpha_i\alpha_j\bigl\vert\sin(\theta _i-\theta_j)\bigr\vert+\alpha_{N+1}\sum_{i=1}^N\alpha_i\bigl\vert\sin(\theta _{N+1} -\theta_i)\bigr\vert\\
&=\frac{1}{2}\sum_{i=1}^{N+1} \sum_{j=1}^{N+1}\alpha_i\alpha_j\bigl\vert\sin (\theta_i-\theta_j)\bigr\vert,
\end{align*}
which proves \eqref{eq:airPolygone}.\qedhere
\end{enumerate}
\end{proof}

In the following, we use a regular subdivision $\theta$. We will show
that if the subdivision step is sufficiently small, then any symmetric
convex set can be approximated by a zonotope as precisely as we want.

\begin{definition}[$0$-regular zonotopes]
For $N\in\mathbb{N}^*$, let $\mathcal{C}^{(N)}_0$ denote the class of
all zonotopes with at most $2N$ faces oriented by the regular
subdivision of $[0,\pi[$ by $N$ elements:
\begin{equation*}
\mathcal{C}^{(N)}_0=\Biggl\lbrace\bigoplus
_{i=1}^N \alpha_i S_{\theta
_i}\big\vert
\alpha\in\mathbb{R}_+ ^N\Biggr\rbrace\text{\quad with }
\theta_i=\frac
{(i-1)\pi}{N}, \ i=1,\dots,N.
\end{equation*}
Such zonotopes are called $0$-regular zonotopes.
\end{definition}
Note that $\mathcal{C}^{(N)}_0\subset\mathcal{C}^{(N)}$ and $\mathcal
{C}^{(N_1)}_0\subset\mathcal{C}^{(N_2)}_0 $ if and only if $N_1$ is a
splitter of $N_2$. In addition, $\mathcal{C}^{(N)}_0$ can be identified
to $\mathbb{R}_+ ^N $ by the mapping $\alpha\rightarrow X=(\bigoplus_{i=1}^N \alpha_i S_{\theta_i}) $, which is an isomorphism between the
semigroups $(\mathbb{R}_+ ^N , +)$ and $(\mathcal{C}^{(N)}_0, \oplus)$.
That is, this mapping is a bijection, and
\begin{align*}
\forall\bigl(\alpha,\alpha'\bigr)\in\mathbb{R}_+ ^N
\times\mathbb{R}_+ ^N, \quad  \Biggl(\bigoplus
_{i=1}^N \bigl(\alpha_i+
\alpha'_i\bigr) S_{\theta_i}\Biggr)=\Biggl(\bigoplus
_{i=1}^N \alpha_i
S_{\theta_i}\Biggr)\oplus\Biggl(\bigoplus_{i=1}^N
\alpha'_i S_{\theta_i}\Biggr).
\end{align*}

\begin{theorem}[Approximation in $\mathcal{C}^{(N)}_0$]
Let $X\in\mathcal{C}$.
\begin{enumerate}
\item For all $N>1$, let $F^{(N)}$ denote the squared matrix $(\vert
\sin(\theta_i-\theta_j)\vert)_{1\leq i,j\leq N}$ and
$H^{(N)}_X=^t(H_X(\theta_1),\dots, H_X(\theta_N))$. Then
\begin{equation}
X_0^{(N)} =\bigoplus^N_{i=1}
\bigl({F^{(N)}}^{-1}H^{(N)}_X
\bigr)_i S_{\theta_i} \label{eq:prop:N-0-approxExpression}
\end{equation}
belongs to $\mathcal{C}^{(N)}_0$ and satisfies
\begin{equation}
\forall N>1, \quad d_H\bigl(X,X_0^{(N)}\bigr)
\leq(6+2\sqrt{2})\sin\biggl(\frac{\pi}{2N}\biggr)\operatorname{diam}(X), \label{eq:majorApprox}
\end{equation}
where $\operatorname{diam}(X)=\sup_{s\in\mathbb{R}}(H_X(s))$ denotes the maximal
diameter of $X$ and $d_H$ the Hausdorff distance.

Consequently, the sequence of $0$-regular zonotopes $(X_0^{(N)})_{N>1}$
approximates $X$ in the following sense:
\begin{equation}
d_H\bigl(X,X_0^{(N)}\bigr)\longrightarrow0\quad
\text{as}\; N\longrightarrow\infty. \label{eq:prop:limites}
\end{equation}
We call $X_0^{(N)}$ the $\mathcal{C}^{(N)}_0$-approximation of $X$.
\item In addition, for any $N>1$, the set $X_0^{(N)}$ is the unique
element of $\mathcal{C}^{(N)}_0$ satisfying
\begin{equation}
H_{X_0^{(N)}}(\theta_i)=H_X(\theta_i),\quad
i=1,\dots,N.
\end{equation}
\item Furthermore, $X_0^{(N)}$ contains $X$ and can be expressed as
\begin{equation}
X_0^{(N)}=\bigcap_{i=1}^N
\biggl\lbrace x\in\mathbb{R}^2,\; \bigl\vert\bigl\langle x,^t\bigl(-\sin (
\theta_i),\cos(\theta_i)\bigr)\bigr\rangle\bigr\vert\leq
\frac{1}{2}H_X(\theta_i)\biggr\rbrace.
\label{eq:prop:N-0-approxExpression2}
\end{equation}
\label{thm:approx}
\end{enumerate}
\end{theorem}

\begin{proof}
$\:$
\begin{enumerate}
\item[2.] For integer $N>1$, it is easy to see that the matrix
$F^{(N)}$ is invertible since $F^{(N)}$ is a circulant matrix \cite
{CirculantMatrix} and its eigenvalues are exactly the coefficients of
the discrete Fourier transform \cite{sundararajan2001discreteFourier}
of the signal $\vert\sin(\cdot)\vert$ (these coefficients are all
strictly positive).
Let $\alpha= {F^{(N)}}^{-1}H^{(N)}_X$ be such that
\begin{equation*}
X_0^{(N)}=\bigoplus^N_{i=1}
\alpha_i S_{\theta_i}.
\end{equation*}
Let us show that $X_0^{(N)}$ is the unique element of $\mathcal
{C}^{(N)}_0$ satisfying $H_{X_0^{(N)}}(\theta_i)=H_X(\theta_i)$, $
i=1,\dots,N $.
Suppose that there exists $X'\in\mathcal{C}^{(N)}_0$ satisfying
$H_{X'}(\theta_i)=H_X(\theta_i),\; i=1,\dots,N $. Then $X'$ can be
written as $X'=\bigoplus^N_{i=1}\alpha'_i S_{\theta_i}$, and then
$H^{(N)}_{X}={F^{(N)}}\alpha' $. The invertibility of $F^{(N)}$
implies $\alpha=\alpha'$, which means that $X_0^{(N)}=X'$.

\item[1.] Let us find an upper bound for the Hausdorff distance.

For all $\eta\in\mathbb{R}$, there exists $i\in\lbrace1,\dots,N\rbrace$
such that $\eta=\theta_i+\delta$ with $\vert\delta\vert\leq\frac{\pi}{2N}$. Using inequality \eqref{Cond:convexPolygon} with $\theta
=\theta_i$ and $\beta=\delta$ for $X_0^{(N)}$, we have
\begin{align*}
& H_{X_0^{(N)}}(\eta)\leq H_{X_0^{(N)}}(\theta_i)+2\biggl\vert\sin
\biggl(\frac
{\delta}{2}\biggr)\biggr\vert H_{X_0^{(N)}}\biggl(\theta_i+
\frac{\delta+\pi}{2}\biggr).
\end{align*}
Using inequality \eqref{Cond:convexPolygon} with $\theta=\eta$ and
$\beta=-\delta$ for $X$, we have
\begin{align*}
& H_{X}(\theta_i)\leq H_{X}(\eta)+2\biggl\vert\sin
\biggl(\frac{-\delta}{2}\biggr)\biggr\vert H_{X}\biggl(\theta_i+
\frac{\delta+\pi}{2}\biggr)
\\
\Rightarrow \quad &-H_{X}(\eta)\leq -H_{X}(\theta_i)
+2\biggl\vert\sin\biggl(\frac
{\delta}{2}\biggr)\biggr\vert H_{X}\biggl(
\theta_i+\frac{\delta+\pi}{2}\biggr).
\end{align*}
Considering the equality $ H_{X_0^{(N)}}(\theta_i)= H_{X}(\theta_i) $,
from the two previous inequalities it follows that
\begin{align*}
&H_{X_0^{(N)}}(\eta)\,{-}\,H_{X}(\eta)\,{\leq}\,2\biggl\vert\sin\biggl(
\frac{\delta}{2}\biggr)\biggr\vert \biggl( H_{X}\biggl(
\theta_i\,{+}\,\frac{\delta+\pi}{2}\biggr)\,{+}\, H_{X_0^{(N)}}\biggl(
\theta_i\,{+}\,\frac
{\delta+\pi}{2}\biggr)\biggr).
\end{align*}
In the same manner, using \eqref{Cond:convexPolygon} with $\theta=\theta
_i$ and $\beta=\delta$ for $X$ and with $\theta=\eta$ and $\beta=-\delta
$ for $X_0^{(N)}$, we have
\begin{align*}
&H_{X}(\eta) \,{-}\,H_{X_0^{(N)}}(\eta)\,{\leq}\,2\biggl\vert\sin\biggl(
\frac{\delta}{2}\biggr)\biggr\vert \biggl( H_{X}\biggl(
\theta_i\,{+}\,\frac{\delta+\pi}{2}\biggr)\,{+}\, H_{X_0^{(N)}}\biggl(
\theta_i\,{+}\,\frac
{\delta\,{+}\,\pi}{2}\biggr)\biggr).
\end{align*}
Therefore, by denoting $\operatorname{diam}(X)=\sup_\theta\lbrace H_X(\theta)\rbrace$
and $\operatorname{diam}(X_0^{(N)})=\sup_\theta\lbrace H_{X_0^{(N)}}(\theta)\rbrace$
it follows that
\begin{equation}
\bigl\vert H_{X}(\eta) -H_{X_0^{(N)}}(\eta)\bigr\vert\leq2\sin\biggl(
\frac{\pi}{2N}\biggr) \bigl( \operatorname{diam}(X)+\operatorname{diam}\bigl(X_0^{(N)}
\bigr)\bigr). \label{eq:major1}
\end{equation}
Furthermore,
\begin{align*}
H_{X_0^{(N)}}(\eta) &=\sum_{j=1}^N\alpha_j\bigl\vert\sin(\theta_i +\delta -\theta_j)\bigr\vert\\
&=\sum_{j=1}^N\alpha_j\bigl\vert\sin(\theta_i -\theta_j)\cos(\delta) -\cos (\theta_i -\theta_j)\sin(\delta) \bigr\vert\\
&\leq\bigl\vert\cos(\delta)\bigr\vert\sum_{j=1}^N\alpha_j\xch{\bigl\vert\sin(\theta_i -\theta_j)\bigr\vert}{\bigl\vert\sin(\theta_i -\theta_j)} \,{+}\,\bigl\vert\sin(\delta)\bigr\vert\sum_{j=1}^N\biggl\vert\sin\biggl(\theta_i -\theta_j +\frac{\pi}{2}\biggr)\biggr\vert\\
&\leq\bigl\vert\cos(\delta)\bigr\vert H_{X_0^{(N)}}(\theta_i) + \bigl\vert\sin(\delta )\bigr\vert H_{X_0^{(N)}}\biggl(\frac{\pi}{2}\biggr)\\
&\leq\bigl\vert\cos(\delta)\bigr\vert H_{X}(\theta_i) + \bigl\vert\sin(\delta)\bigr\vert \operatorname{diam}\bigl(X_0^{(N)}\bigr)\\
&\leq\bigl\vert\cos(\delta)\bigr\vert \operatorname{diam}(X)+ \bigl\vert\sin(\delta)\bigr\vert \operatorname{diam}\bigl(X_0^{(N)}\bigr)\\
&\leq \operatorname{diam}(X)+ \sin\biggl(\frac{\pi}{2N}\biggr) \operatorname{diam}\bigl(X_0^{(N)}\bigr)\\
\Rightarrow&\quad  \operatorname{diam}\bigl(X_0^{(N)}\bigr) \biggl(1-\sin\biggl(\frac{\pi}{2N}\biggr)\biggr)\leq \operatorname{diam}(X),\\
N\geq2 \quad \Rightarrow& \quad \operatorname{diam}\bigl(X_0^{(N)}\bigr)\leq\frac{\sqrt{2}}{\sqrt{2}-1} \operatorname{diam}(X).
\end{align*}\vadjust{\eject}
Then from \eqref{eq:major1} we have
\begin{align*}
&\bigl\vert H_{X}(\eta) -H_{X_0^{(N)}}(\eta)\bigr\vert \leq2\sin\biggl(\frac{\pi}{2N}\biggr) \biggl(1+\frac{\sqrt{2}}{\sqrt{2}-1}\biggr) \operatorname{diam}(X)\\
\Rightarrow\quad & \sup_\eta\bigl\vert\bigl(H_{X}(\eta)-H_{X_0^{(N)}}(\eta)\bigr)\bigr\vert =d_H\bigl(X,X_0^{(N)}\bigr)\\
&\quad \leq(6+2\sqrt{2})\sin\biggl(\frac{\pi}{2N}\biggr)\operatorname{diam}(X).
\end{align*}
Consequently, $d_H(X,X_0^{(N)})\longrightarrow0\; \text{as}\;
N\longrightarrow\infty$.
\item[3.] Note that $Y_N=\bigcap_{i=1}^N\lbrace x\in\mathbb{R}^2,\;
\vert\<x,^t(
-\sin(\theta_i),\cos(\theta_i))\>\vert\leq\frac{1}{2}H_X(\theta
_i)\rbrace$. Then $Y_N\in\mathcal{C}^{(N)}_0$. Indeed, each set of the
intersection is the space between two lines oriented by one of the
$\theta_i$; thus, $Y_N$ is a polygon with faces directed by the $\theta
_i$, and therefore it belongs to $\mathcal{C}^{(N)}_0$. Because of the
symmetry of $X$, it is easy to see that $X=\bigcap_{s\in[0,\pi]}\lbrace
x\in\mathbb{R}^2,\; \vert\<x,^t(-\sin(s),\cos(s))\>\vert\leq\frac
{1}{2}H_X(s)\rbrace$; therefore, $ X\subset Y_N$, and consequently
$H_X\leq H_{Y_N}$. Furthermore, because of the expression of $Y_N$ for
any $i=1,\dots,N$, $H_X(\theta_i)\geq H_{Y_N}(\theta_i)$ with the
equality on $\theta_i$, and according to the foregoing,
$Y_N=X_0^{(N)}$.\qedhere
\end{enumerate}
\end{proof}

This theorem shows that a symmetric body can be always approximated by
a $0$-regular zonotope as close as we want. Note that the choice of the
sequence $X_0^{(N)}$ is not the best one. Indeed, by taking $\frac
{\operatorname{diam}(X)}{\operatorname{diam}(X_0^{(N)})} X_0^{(N)}$ there is a finer approximation
with respect to the Hausdorff distance. However, the sequence
$X_0^{(N)}$ presents some important advantages: it always contains $X$,
the approximation of a Minkowski sum is the Minkowski sum of the
approximations, and its face length vector is expressed only from a
linear combination of the Feret diameter of $X$. Furthermore, if there
exists $M>1$ such that $X\in\mathcal{C}^{(M)}_0$, then $X_0^{(M)}=X$, and $X$ is an adhesion value of the sequence $X_0^{(N)}$.

\begin{remark}[Equivalence between perimeter and maximal diameter]
Notice that $\operatorname{diam}(X)$ can be replaced by $\frac{1}{2}U(X)$ in relation
\eqref{eq:majorApprox}. In fact, for any convex set $X$, we have the relation
\begin{equation}
2 \operatorname{diam}(X)\leq U(X)\leq4 \operatorname{diam}(X). \label{eq:equivalenceDiamPer}
\end{equation}
Indeed, according to the definition of $\operatorname{diam}(X)$, there exists a line
segment $S\subseteq X$ that has the length greater than $\operatorname{diam}(X)$, and
then $U(X)\geq U(S)\geq2 \operatorname{diam}(X)$. The second inequality comes by
considering that there is a square of side $\operatorname{diam}(X)$ containing $X$.
\end{remark}

\begin{remark}[Expression of the mixed area from the Feret diameter]
\label{rq:mixedareaapprox}
An interpretation of the mixed area between a convex set and a
symmetric convex set can be given from Theorem~\ref{thm:approx}.
Indeed, let $N>1$, $Y$ be a convex set \textup{(}not necessarily
symmetric\textup{)}, $X$ be a symmetric convex set, and
$X_0^{(N)}=\bigoplus^N_{i=1}\alpha_i S_{\theta_i}$ be its $\mathcal
{C}^{(M)}_0$-approximation. Then, according to the continuity of the
area and the Minkowski addition, there is
\begin{align*}
W\bigl(Y,X_0^{(N)}\bigr)\rightarrow W(Y,X)\text{\quad as } N
\rightarrow\infty.
\end{align*}
Furthermore, according to Theorem~\ref{thm:approx}, $W(Y,X_0^{(N)})$
can be expressed as
\begin{align*}
W\bigl(Y,X_0^{(N)}\bigr)=\sum
_{i=1}^N H_Y(\theta_i)\sum
_{j=1}^N {F^{(N)}_{ij}}^{-1}
H_X(\theta_j).
\end{align*}
Then, the mixed area $ W(Y,X)$ can be computed as
\begin{align*}
W(Y,X)= \underset{N\rightarrow\infty} {\lim} \sum_{i=1}^N
\sum_{j=1}^N {F^{(N)}_{ij}}^{-1}
H_Y(\theta_i)H_X(\theta_j).
\end{align*}
Notice that a continuous version of this expression can be written in
terms of convolution. However, this is not our objective.
\end{remark}

Of course, the $\mathcal{C}^{(N)}_0$-approximation is sensitive to
rotations (see Fig.~\ref{fig:worstandbetter}). Obviously, it can be
problematic to describe the geometry of sets. Let us consider the
following example of an ellipse.

\begin{figure}[t]
\includegraphics{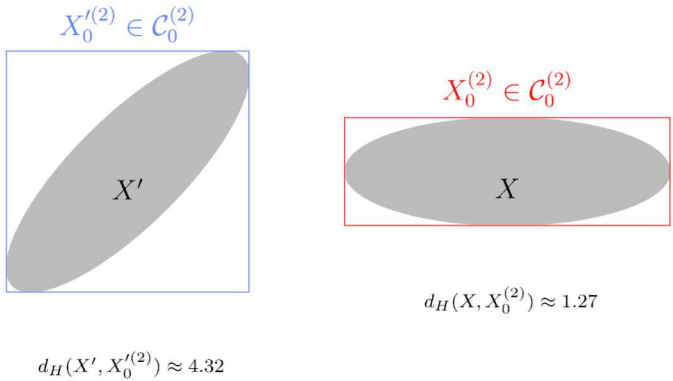}
\caption{The $\mathcal{C}^{(N)}_0$-approximations of an ellipse $X$ and
its rotation $X'$ with respect to the angle~$\frac{\pi}{4}$}
\label{fig:worstandbetter}
\end{figure}

\begin{example}
Let $X$ be an ellipse with semiaxis $a=1$ and $b=3$, and suppose that
the major semiaxis $b$ is horizontally oriented. Firstly, consider the
case $N=2$, and let us denote $X':=R_{\frac{\pi}{4}}(X)$, Fig.~\ref
{fig:worstandbetter} shows that the $\mathcal{C}^{(N)}_0$-approximation
of $X$ is better than that of $X'$ \textup{(}in terms of the Hausdorff
distance\textup{)}. Indeed, $d_H(X,X^{(2)}_0) \ll d_H(X',X'^{(2)}_0 )$.
Furthermore, the $\mathcal{C}^{(2)}_0$-approximation of the rotation is
not the rotation of the $\mathcal{C}^{(2)}_0$-approximation. Therefore,
it can be problematic to use the $\mathcal{C}^{(2)}_0$-approximation to
describe the shape of $X$. Note that for the ellipse $X$ of Fig.~\ref
{fig:worstandbetter}, the orientations $0$ and $\frac{\pi}{4}$ are
respectively the better and the worst cases for the $\mathcal
{C}^{(2)}_0$-approximation.

Let us consider now the more general case of the approximation of the
rotations of $X$ for different values of $N$. For each $N=1,\dots,20$,
the $\mathcal{C}^{(N)}_0$-approximations of all of the rotations $R_\eta
(X)$ of $X$ have been computed. Among these approximations, the better
$\eta_b$ and the worst $\eta_w$ angles \textup{(}in terms of the
Hausdorff distance\textup{)} have been retain. The corresponding
Hausdorff distances are represented in Fig.~\ref{fig:betterandworstN}.
Consequently, whatever the orientation of the ellipse, the Hausdorff
distance is inside the gray region. We can be notice that, for small
values of $N$, the difference between the worst and the better case is
more important.
\end{example}

\begin{figure}[t]
\includegraphics{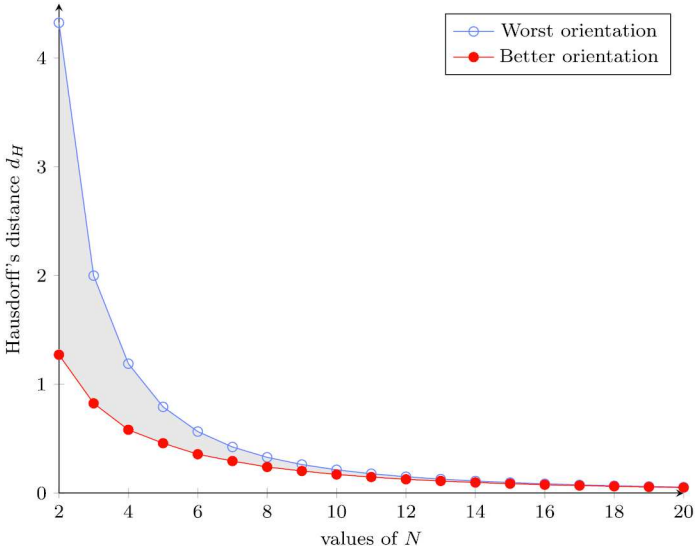}
\caption{The Hausdorff distance between an ellipse of semiaxis $(1,3)$
and its $\mathcal{C}^{(N)}_0$-approximation for several values of $N$.
The case of the better direction in red and worst in blue. The gray
region represents the possible values for this distance}
\label{fig:betterandworstN}
\end{figure}
For the reasons mentioned, it can be interesting to have an isometric
invariant approximation. Fortunately, for a symmetric convex set $X$,
the better $\mathcal{C}^{(N)}_0$-approximation (in terms of the
Hausdorff's distance) of the family of rotations of $X$ can be used to
define such an isometric invariant approximation.

\subsection{Approximation of a symmetric convex set by a regular zonotope}

We have shown previously how a symmetric convex set $X$ can be
approximated in the class of $0$-regular zonotopes. Such an
approximation is sensitive to the rotations. However, in order to study
convex sets, there is sometimes a need to have isometric invariant
tools. Therefore, we will define here an approximation that is
invariant up to a rotation. To meet this goal, there is a need to
perform the approximation on a class larger than $\mathcal{C}^{(N)}_0$,
namely the class of regular zonotopes.

\begin{definition}[$t$-regular and regular zonotopes]
Let $t\in\mathbb{R}$, $N>1$ be an integer, and let $\mathcal
{C}^{(N)}_t$ denote the class of the rotated elements of $\mathcal
{C}^{(N)}_0$ with respect to the angle~$t$:
\begin{equation*}
\mathcal{C}^{(N)}_t=\bigl\lbrace R_t(X)\bigl\vert X
\in\mathcal{C}^{(N)}_0\bigr\rbrace.
\end{equation*}
Any element of $\mathcal{C}^{(N)}_t$ is called a \textit{$t$-regular
zonotope with $2N$ faces}.

Furthermore, $\mathcal{C}^{(N)}_\infty=\bigcup_{t\in\mathbb{R}} \mathcal
{C}^{(N)}_t$ denotes the set of \textit{regular zonotopes with $2N$ faces}.
\end{definition}
All the properties of $\mathcal{C}^{(N)}_0$ cited before are also true
for $\mathcal{C}^{(N)}_t,\;t\in\mathbb{R}$. Therefore, we will define
an approximation in $\mathcal{C}^{(N)}_\infty$.

\begin{theorem}[Approximation in $\mathcal{C}^{(N)}_\infty$]
Let $X\in\mathcal{C}$, and let us denote by $X_0^{N}(t)$ the $\mathcal
{C}^{(N)}_0$-approximation of $R_{-t}(X)$.
\begin{enumerate}
\item There exists $\tau\in[0,\pi[$ satisfying
\begin{equation}
d_H\bigl(R_{\tau}\bigl(X_0^{N}(
\tau)\bigr),X\bigr)=d_H\bigl(X_0^{N}(
\tau),R_{-\tau}(X)\bigr)=\min_{t\in\mathbb{R}}d_H
\bigl(X_0^{N}(t),R_{-t}(X)\bigr). \label{eq:thm2:defN-infty}
\end{equation}
We call $X_0^{N}(\tau)$ \textup{(}also denoted $\tilde{X}_0^{N})$ the
\textit{$\mathcal{C}^{(N)}_0$-rotational approximation of~$X$}.
\item The $\mathcal{C}^{(N)}_0$-rotational approximation of $X$ is
invariant under rotations of~$X$.
\end{enumerate}
The set $ R_{\tau}(X_0^{N}(\tau))$ is called a \textit{$\mathcal
{C}^{N}_\infty$-approximation of $X$ in $\mathcal{C}_{\infty}^{(N)}$}
and is denoted by $X^{(N)}_\infty$.
\label{thm:2approxuptorot}
\end{theorem}

\begin{proof}
$\:$
\begin{enumerate}
\item[1.] First of all, because of the symmetry of the $0$-regular zonotopes,
\begin{align*}
&\forall t\in\mathbb{R},\quad \mathcal{C}^{(N)}_t=\mathcal{C}^{(N)}_{t+\pi}\\
\Rightarrow\quad & \min_{t\in\mathbb{R}}d_H\big(X_0^{N}(t),R_{-t}(X)\big)=\min_{t\in[0,\pi]}d_H\big(X_0^{N}(t),R_{-t}(X)\big).
\end{align*}
For any $t\in\mathbb{R}$, let us denote by $\alpha(t)$ the face length
vector of $X_0^{N}(t)$. Then, for any $h\in\mathbb{R}$,
\begin{align*}
&\big\Vert\alpha(t)-\alpha(t+h)\big\Vert_{_1} =\big\Vert{F^{(N)}}^{-1}\bigl( H^{(N)}_{R_{-t}(X)}-H^{(N)}_{R_{-t-h}(X)}\bigr)\big\Vert _{_1}\\
\Rightarrow\quad &\big\Vert\alpha(t)-\alpha(t+h)\big\Vert_{_1} \leq\big\Vert{F^{(N)}}^{-1}\big\Vert_{_1}\big\Vert H^{(N)}_{R_{-t}(X)}-H^{(N)}_{R_{-t-h}(X)}\big\Vert_{_1}.
\end{align*}
However, $\forall\eta\in\mathbb{R},\; H_{R_{-t-h}(X)}(\eta
)=H_{R_{-t}(X)}(\eta+h)$. Because of the continuity of the Feret
diameter, $\Vert H^{(N)}_{R_{-t}(X)}-H^{(N)}_{R_{-t-h}(X)}\Vert
_{_1}\rightarrow0$ as $h\rightarrow0$, and thus $\Vert\alpha
(t)-\alpha(t+h)\Vert_{_1} \rightarrow0$ as $\rightarrow0$.

Therefore, from expression \eqref{eq:ferretPolygone} about the Feret
diameter of a zonotope, for all $\eta\in\mathbb{R}$,
\begin{align*}
\bigl\vert H_{X_0^{N}(t+h)}(\eta)- H_{X_0^{N}(t)}(\eta)\bigr\vert&= \Biggl\vert\Biggl(\sum_{i=1}^N \bigl(\alpha_i(t)-\alpha_i(t+h)\bigr)\bigl\vert\sin(\eta-\theta_i)\bigr\vert\Biggr)\Biggr\vert\\
&\leq N\underset{i=1,\dots, N} {\max}\bigl\lbrace\bigl(\alpha_i(t)-\alpha _i(t+h)\bigr)\bigr\rbrace.
\end{align*}
Therefore, $\vert H_{X_0^{N}(t+h)}(\eta)- H_{X_0^{N}(t)}(\eta)\vert
\rightarrow0$ as $h\rightarrow0$, and, finally,
$d_H(X_0^{N}(t),X_0^{N}(t+h))\rightarrow0$ as $h\rightarrow0$.
Consequently, the map $t\mapsto X_0^{N}(t)$ is continuous with respect
to the Hausdorff distance.

Note that for all $x\in\mathbb{R},\; H_ {R_t(X_0^{N}(t))}(x)=H_
{X_0^{N}(t)}(x-t)$ and $H_X(x)=H_{R_{-t}(X)}(x-t)$. Then
\begin{align*}
&H_ {R_t(X_0^{N}(t))}(x)-H_X(x)= H_
{X_0^{N}(t)}(x-t)-H_{R_{-t}(X)}(x-t)
\\
\Rightarrow\quad &  d_H\bigl(R_t\bigl(X_0^{N}(t)
\bigr),X\bigr)=d_H\bigl(X_0^{N}(t),R_{-t}(X)
\bigr)
\\
\Rightarrow\quad & \min_{t\in\mathbb{R}}d_H\bigl(R_t
\bigl(X_0^{N}(t)\bigr),X\bigr)=\min_{t\in
\mathbb{R}}d_H
\bigl(X_0^{N}(t),R_{-t}(X)\bigr).
\end{align*}
Furthermore, for any $x,h\in\mathbb{R}$,
\begin{align*}
&\bigl\vert H_ {R_t(X_0^{N}(t))}(x)- H_ {R_{t+h}(X_0^{N}(t+h))}(x)\bigr\vert \\
&\quad =\bigl\vert H_ {X_0^{N}(t)}(x-t)\\
&\qquad-\cdots-H_ {X_0^{N}(t+h)}(x-t)+ H_ {X_0^{N}(t+h)}(x-t)-H_{X_0^{N}(t+h)}(x-t-h)\bigr\vert\\
&\quad \leq\bigl\vert H_ {X_0^{N}(t)}(x-t)- H_ {X_0^{N}(t+h)}(x-t)\bigr\vert\\
&\qquad+\cdots+\bigl\vert H_ {X_0^{N}(t+h)}(x-t)- H_ {X_0^{N}(t+h)}(x-t-h)\bigr\vert.
\end{align*}
Then from the continuity of the Feret diameter and of the map $t\mapsto
X_N(t)$ there follows the continuity of $t\mapsto R_t(X_0^{N}(t))$. As
a consequence, the map $t\mapsto d_H(X_0^{N}(t),X)$ is also continuous,
and the minimum $\min_{t\in[0,\pi]}d_H(R_t(X_0^{N}(t)),X) $ is
achieved. Then there is $a\in[0,\pi]$ such that $d_H(R_\tau
(X_0^{N}(\tau)),X)=\min_{t\in\mathbb{R}}d_H(X_0^{N}(t),R_{-t}(X))$.
\item[2.] Let us prove the invariance by rotations. Let $\eta\in[0,\pi
]$ and $Y=R_\eta(X)$. Then $Y_0^{N}(t)$ is the $\mathcal
{C}_{0}^{(N)}$-approximation of $R_{-(t-\eta)}(X)$, and
$Y_0^{N}(t)=X_0^{N}(t-\eta)$. Furthermore,
\begin{align*}
\min_{t\in\mathbb{R}}d_H\bigl(Y_0^{N}(t),R_{-t}(Y)
\bigr)&=\min_{t\in\mathbb
{R}}d_H\bigl(X_0^{N}(t-
\eta),R_{-(t-\eta)}(X)\bigr)
\\
&=\min_{t\in\mathbb{R}}d_H\bigl(X_0^{N}(t),R_{-(t)}(X)
\bigr)
\\
&=d_H\bigl(X_0^{N}(\tau),R_{-\tau}(X)
\bigr).
\end{align*}
Then $X_0^{N}(\tau)$ is a $\mathcal{C}^{(N)}_0$-rotational
approximation of $Y$, and the $\mathcal{C}^{(N)}_\infty$-approximation
associated is $ R_{\eta}(R_\tau(X_0^{N}(\tau)))$ (indeed, $Y_0^{N}(\tau
+\eta)=X_0^{N}(\tau)$).\qedhere
\end{enumerate}
\end{proof}

The theorem gives important information. The $\mathcal{C}^{(N)}_\infty
$-approximation of a symmetric convex set $X$ is the best regular
zonotope with at most $2N$ faces containing $X$. It is always a better
approximation than the $\mathcal{C}^{(N)}_0$-approximation. This
approximation can be used for not so large $N$. For example, for $N=2$,
the $\mathcal{C}^{(2)}_0$-approximation of an ellipse depends on the
orientation of the ellipse, but its $\mathcal{C}^{(2)}_\infty
$-approximation is the best way to put the ellipse inside a rectangle
(see Fig.~\ref{fig:2}). An illustration of the approximations of that
ellipse for higher values of $N$ is represented Fig.~\ref
{fig:approxC0etCinf3410}.

\begin{figure}[!h]
\includegraphics{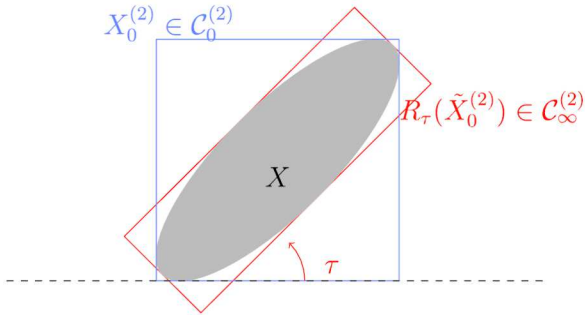}
\caption{An ellipse and its approximations: $X_2\in\mathcal{C}^{(2)}_0$
in blue and $R_\tau(\tilde{X}_2)\in\mathcal{C}^{(2)}_\infty$ in red}
\label{fig:2}
\end{figure}
\begin{figure}[h!]
\includegraphics{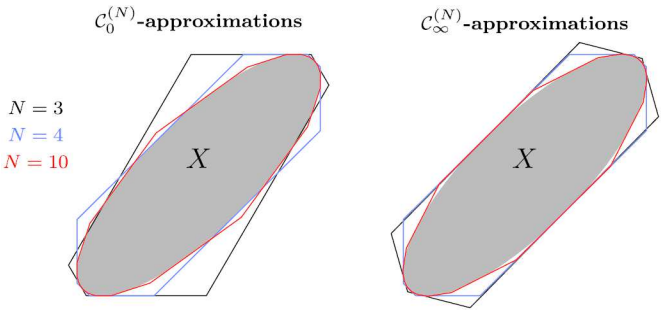}
\caption{The $\mathcal{C}^{(N)}_0$-approximations (left) and $\mathcal
{C}^{(N)}_\infty$-approximations (right) of an ellipse of semiaxis
$(3,1)$ for different values of $N(=3,4,10)$}

\label{fig:approxC0etCinf3410}
\end{figure}

The accuracy of the $\mathcal{C}^{(N)}_0$-approximation is presented in
Fig.~\ref{fig:betterandworstN}, and we remark that the best orientation
corresponds to the $\mathcal{C}^{(N)}_\infty$-approximation. Then, for
the considered ellipse, the accuracy of the $\mathcal{C}^{(N)}_\infty
$-approximation in function of the number of faces $N$ is represented
in Fig.~\ref{fig:betterandworstN}. However, the accuracy of the
$\mathcal{C}^{(N)}_\infty$-approximation depends on both shape and size
of the symmetric convex set $X$.
\begin{figure}[h!]
\includegraphics{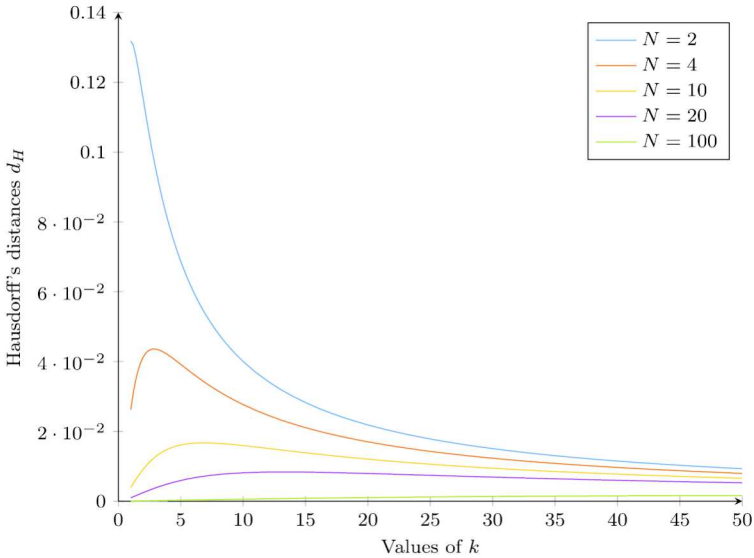}
\caption{The Hausdorff distance between an ellipse of unit perimeter
and its $\mathcal{C}^{(N)}_\infty$-approximations for several values of
$N$ in function of its axis ratio $k$}
\label{fig:accCinf}
\end{figure}

\begin{remark}[Accuracy of the $\mathcal{C}^{(N)}_\infty $-approximation]
The size dependence of the accuracy is easy to understand: the accuracy
decreases proportionally to the size factor. Indeed, for $Y:=kX,\;k\in
\mathbb{R}_+$, we have $d_H(Y_\infty^{(N)},Y)=k d_H(X_\infty^{(N)},X)$
\textup{(}because of the homogeneity of the Feret diameter\textup{)}.
In order to study the impact of the shape \textup{(}independently of
its size\textup{)} on the approximation accuracy, we need to use a
homothetic invariant descriptor. In order to do this, we normalize the
Feret diameter of a symmetric convex set $X$ by its perimeter.
According to Cauchy's formula \textup{\cite{schneider2013convex}}, the
perimeter is equal to the Feret diameter total mass $\int_0^\pi
H_X(\theta)d\theta$. Then, according to the homogeneity of the Feret
diameter, an involved distance can be defined as $\tilde
{d}_H(X,Y):=d_H(\frac{X}{U(X)},\frac{Y}{U(Y)})$ for all $X,Y \in\mathcal
{C}$. Such a distance can be used to study the approximation accuracy.
Notice that it is equivalent to work with sets of unit perimeters and
using the usual Hausdorff distance. Such a consideration will be done
in the following example.

Let us consider an ellipse $X$ with unit perimeter and axis ratio $k\in
[1,+\infty[$, the case $k=1$ referring to the disk. The accuracy of the
$\mathcal{C}^{(N)}_\infty$-approximation as a function of $N$ and $k$
is shown. More specifically, on the Fig.~\ref{fig:accCinf}, we can see
that the behavior of the curves is very different for different values
of $N$. Indeed, the worst shape for $N=2$ is the disk. However, as we
can see that this is not the case for other values of $N$. We can
notice that when the ratio $k$ increases, the importance of $N$ for the
approximation decreases. This suggests that when an object $X$ is
elongated, we can choose a small value of $N$.
\end{remark}

We have studied two different approximations of a symmetric convex set~$X$. The first one is an approximation of~$X$ as a $0$-regular
zonotope, and the second as a regular zonotope. These approximations
have been characterized from the Feret diameter of~$X$. The next
objective is to study these approximations when~$X$ becomes a random
symmetric body, and then how they can be characterized from the Feret
diameter of~$X$. In order to do this, we need to study some properties
of the random zonotopes, which lead us to the following section.

\section{The random zonotopes}

The aim of this section is to investigate how a random zonotope can be
described by a random vector representing its faces and how such a
random vector can be characterized from the Feret diameter of the
random zonotope. Firstly, we will investigate the properties of the
random process corresponding to the Feret diameter of a random set.
Secondly, we will explore the description of a random zonotope by its
faces. Finally, we will give a characterization of some random
zonotopes from their Feret diameter random process.

\subsection{Feret diameter process and isotropic random set}
Let $X$ be a random convex set, that is, a random closed set that is
almost surely a convex set. In this subsection, we state some
properties of the random process \cite{cox1977theoryRandomProcess}
corresponding to the Feret diameter of $X$.
\begin{definition}[Feret diameter random process]
Let $X$ be a random convex set of $\mathbb{R}^2$. For \(P\)-almost all
$\omega\in\varOmega, \text{ } X(\omega)$ is a convex set. Then, for any
$t\in\mathbb{R}$, the positive random variable $H_X(t):\omega\mapsto
H_{X(\omega)}(t)$ is almost surely defined. The random process $
\lbrace H_X(t),t\in\mathbb{R}\rbrace$ is called the \textit{Feret
diameter random process} of~$X$.
\end{definition}
The trajectories of $H_X$ are the Feret diameter of the realizations of
$X$. The properties in Proposition~\ref{prop:feretprop} are also true
for these trajectories, in particular, the continuity and $\pi
$-periodicity. We can also notice that the Feret diameter random
process characterizes the symmetric convex sets.

\begin{definition}[Isotropized set of a random symmetric body]
Let $X'$ be a symmetric random convex set, and let $\eta$ be a random
uniform variable on $[0,\pi]$ independent of~$X'$. Then the set
\begin{equation*}
X:=R_\eta\bigl(X'\bigr)
\end{equation*}
is isotropic (a random compact is said to be isotropic if and only if
its distribution is isometric invariant \cite{chiu2013stochastic}) and
is called an isotropized set of $X'$.
\label{def:isotroVersion}
\end{definition}
Let $X'$ be a random symmetric body, and $X$ be an isotropized set of
it. Then $X$ and $X'$ have the same shape distribution and the same
zonotope rotational approximations (see Theorem~\ref{thm:2approxuptorot}).

In the following, we will show that the Feret diameter random process
$H_{X'}$ of $X'$ can be expressed from that of $X$. We will use this
property to show that a random symmetric convex set can be described up
to a rotation by an isotropic random zonotope.

Let us recall that the Feret diameter random process $H_{X'}$ of $X'$
is sufficient to characterize $X'$. Then, for any $\theta\in\mathbb
{R}$, the Feret diameter $H_{X'}$ of $X'$ can be expressed as
\begin{equation*}
H_{X}(\theta)=H_{X'}(\theta-\eta).
\end{equation*}
Let $ B$ be a Borel subset of $\mathbb{R}$. Because of the uniformity
of $\eta$ and its independence from $X'$, it follows that
\begin{align*}
\mathbb{P}\bigl( H_{X}(\theta)\in B\bigr) &=\mathbb{P}
\bigl(H_{X'}(\theta-\eta)\in B\bigr)
\\
&=\frac{1}{2\pi}\int_0^{2\pi}\mathbb{P}
\bigl(H_{X'}(\theta-t)\in B\bigr)\,dt.
\end{align*}
Furthermore, by using the $\pi$-periodicity of the Feret diameter the
distribution of $H_{X}(\theta)$ can be expressed as
\begin{equation}
\mathbb{P}\bigl( H_{X}(\theta)\in B\bigr)=\frac{1}{\pi}\int
_0^{\pi}\mathbb {P}\bigl(H_{X'}(\theta-t)
\in B\bigr)\,dt. \label{eq:distrHisotropised}
\end{equation}
Consequently, the moments of the Feret diameter process of the set $X'$
and the isotropized set $X$ are related. Of course, we need to ensure
their existence, but we will treat this later.

\begin{proposition}[Moments of the Feret diameter process of
the isotropized set]
Let $X'$ be a random convex set, and $X$ the isotropized set of $X'$.
Suppose that the first- and second-order moments of the Feret diameter
random process $H_{X'}$ of $X'$ exist.Then those of $X$ exist and can
be expressed as follows:
\begin{align*}
&\forall\theta\in[0,2\pi],\quad  \mathbb{E}\bigl[H_{X}(\theta) \bigr]=
\frac{1}{\pi
}\int_0^{\pi}\mathbb{E}
\bigl[H_{X'}(\theta) \bigr]\,d\theta,
\\
&\forall(s,t)\in[0,2\pi]^2,\quad  \mathbb{E}\bigl[H_{X}(s)H_{X}(t)
\bigr]=\frac
{1}{\pi}\int_0^{\pi}\mathbb{E}
\bigl[H_{X'}(\theta)H_{X'}(\theta+s-t) \bigr]\,d\theta.
\end{align*}
\label{prop:isotroVersionMomentsFeret}
\end{proposition}

\begin{proof}
Let $X'$ be a random convex set, and $X=R_\eta(X')$ an isotropized set
of it. Suppose that the first- and second-order moments of $H_{X'}$
exist. Recall that $\;H_{X}(\theta)= H_{X'}(\theta-\eta)$ for all
$\theta\in\mathbb{R}$ and that $\eta$ is independent of $X'$, and thus
the result follows by integrating with respect to the uniform
distribution of $\eta$.
\end{proof}

\begin{proposition}[Feret diameter process of an isotropic random convex set]
Let $X'$ be a random convex set.
\begin{enumerate}
\item If $X'$ is isotropic, then the random variables $H_{X'}(\theta
),\; \theta\in[0,\pi]$, are identically distributed \textup{(}i.e.,
the random process $H_{X'}$ is stationary\textup{)}.
\item Furthermore, if $X'$ is symmetric, then the converse is true.
\end{enumerate}
\label{prop:isotropFeret}
\end{proposition}

\begin{proof}
$\:$
\begin{enumerate}
\item[1.] Let $\eta$ be a uniform random variable on $[0,\pi]$
independent of $X'$ and let note $X=R_\eta(X')$. If $X'$ is isotropic,
then $X$ and $X'$ have the same distribution, so that $H_X$ and
$H_{X'}$ also have the same distribution. Consequently, according to
\eqref{eq:distrHisotropised}, for any $\theta\in[0,\pi]$ and any Borel
set~$B$,
\begin{align*}
\mathbb{P}\bigl( H_{X'}(\theta)\in B\bigr)=\mathbb{P}\bigl(
H_{X}(\theta)\in B\bigr)=\frac
{1}{\pi}\int_0^{\pi}
\mathbb{P}\bigl(H_{X'}(\theta-t)\in B\bigr)\,dt.
\end{align*}
Because of the $\pi$-periodicity of the Feret diameter, the integral is
independent of $\theta$, and thus the random variables $H_{X'}(\theta
),\; \theta\in[0,\pi]$, are identically distributed.
\item[2.] Suppose that $X'$ is symmetric and $H_{X'}(\theta),\; \theta
\in[0,\pi]$, are identically distributed. Then the random process $
H_{X'}$ is stationary, that is, for any $x\in\mathbb{R}$, the random
process $(H_{X'}(\theta))_ {\theta\in\mathbb{R}}$ and the translated
process $(\tilde{H}_{X'}(\theta)=H_{X'}(\theta+x))_ {\theta\in\mathbb
{R}}$ have the same distribution. However, $\tilde{H}_{X'}$ is exactly
the random process corresponding to the Feret diameter of $R_x(X')$. It
has been already established that the Feret diameter characterizes the
symmetric bodies. Therefore, for any $x\in\mathbb{R},R_x(X')\text{ and } X'$
have the same distribution, so that $X'$ is isotropic.\qedhere
\end{enumerate}
\end{proof}

We have shown some properties of the Feret diameter random process. Let
us discuss now the random zonotopes, that is, the random sets almost
surely valued in $\mathcal{C}^{(N)}$.

\subsection{Description of the random zonotopes from their faces}

Here we will define some classes of random zonotopes, in particular,
the class of the random zonotopes almost surely valued in $ \mathcal
{C}^{(N)}_0$ and the class of those almost surely valued in $ \mathcal
{C}^{(N)}_\infty$. We will study several properties of the random
zonotopes. In particular, we will show how a random zonotope can be
described by a random vector corresponding to its faces.
\begin{definition}[Random zonotopes]
For an integer $N>1$, a random closed set $X$ that has realizations
almost surely in $\mathcal{C}^{(N)}$ is called a \textit{random
zonotope with at most $2N$ faces} or, in a more concise way, a \textit
{random zonotope} when there is no possible confusion.
\end{definition}
Such a random set can be described almost surely as
\begin{align*}
\forall\omega\in\varOmega\text{ a.s.},\quad X(\omega)=\bigoplus
_{i=1}^{N}\alpha_i(\omega)S_{\beta_i(\omega)}.
\end{align*}
The distribution of the random vector $(\alpha,\beta)$ characterizes
$X$. The random vector $\alpha$ is called a \textit{face length vector}
of $X$.

According to Proposition~\ref{prop:A P H deterministe}, for any face
length vector $\alpha$ of $X$, some geometrical characteristics (Feret
diameter, perimeter, area) of $X$ can be expressed as:
\begin{align}
\forall\omega\in\varOmega\text{ a.s.},\ \forall t\in\mathbb{R},\quad  H_{X}(t)&=\sum_{i=1}^N \alpha_i\bigl\vert\sin(t-\beta_i )\bigr\vert; \label{eq:FeretRZ}\\
\forall\omega\in\varOmega\text{ a.s.},\quad U(X)&=2\sum_{i=1}^N\alpha_i; \label{eq:PerimeterRZ}\\
\forall\omega\in\varOmega\text{ a.s.},\quad A(X)&=\frac{1}{2}\sum_{i=1}^N \sum_{j=1}^N\alpha_i\alpha_j\bigl\vert\sin(\beta_i-\beta_j)\bigr\vert. \label{eq:AreaRZ}
\end{align}

\begin{proposition}[Existence conditions for the autocovariance
of the Feret diameter process]
Let $X$ be a random zonotope with $2N$ faces, and $\alpha$ its face
length vector. Then the following properties are equivalent:
\begin{equation}
\mathbb{E}\bigl[U(X)^2\bigr]<\infty;
\end{equation}
\begin{equation}
\alpha\in L^2\bigl(\mathbb{R}_+^N\bigr).
\end{equation}
Furthermore, if one of these conditions is satisfied, then $\mathbb
{E}[A(X)]<\infty$, and $\mathbb{E}[H_{X}(s)H_{X}(t) ]<\infty$ for all
$(s,t)\in[0,\pi]^2$.
\label{prop:CondExistMoments}
\end{proposition}

\begin{proof}
According to \eqref{eq:PerimeterRZ}, $U(X)^2=(2\sum_{i=1}^N \alpha
_i)^2$, and the first equivalence is trivial (because of the positivity
of $\alpha$).

Proposition~\ref{prop:A P H deterministe} also shows that, for all $
(s,t)\in[0,\pi]^2,\;$
\begin{align*}
H_{X}(s) H_{X}\bigl(t'\bigr)&=\sum
_{i=1}^N\sum_{j=1}^N
\alpha_i\alpha_j \bigl\vert\sin (s-\eta-\beta_i )
\sin(t-\eta-\beta_i )\bigr\vert
\\
&\leq\sum_{i=1}^N\sum
_{j=1}^N \alpha_i\alpha_j
\\
&\leq\frac{1}{4}U(X)^2.
\end{align*}
Then the expectation $\mathbb{E}[H_{X}(s)H_{X}(t) ]$ exists, and the
existence of $\mathbb{E}[A(X)]$ follows from the isoperimetric inequality.
\end{proof}

\begin{definition}[$0$-regular random zonotopes]
For an integer $N>1$, a~random closed set $X$ that has its realizations
almost surely in $\mathcal{C}^{(N)}_0$ is called a~\textit{$0$-regular
random zonotope with at most $2N$ faces} or, in a more concise way, a
\textit{$0$-regular random zonotope} when there is no possible confusion.
\end{definition}
A $0$-regular random zonotope $X$ can be almost surely expressed as
\begin{align*}
\forall\omega\in\varOmega\text{ a.s.},\quad X(\omega)=\bigoplus
_{i=1}^{N}\alpha_i(\omega)S_{\theta_i},
\end{align*}
where $\theta_i,\; i=1,\dots,N$, denotes the regular subdivision on
$[0,\pi]$.

The distribution of the face length vector $\alpha$ characterizes the
distribution of~$X$. In addition, this relation is bijective; in other
word, the distribution of~$\alpha$ is uniquely defined and is called
the face length distribution.

Of course, the $0$-regular random zonotopes can be used to approximate
the random symmetric convex sets as $N\rightarrow\infty$ (see Section~\ref{sec:approx0reg}). However, it is not the best way to model a random
symmetric convex set. Indeed, notice that a~$0$-regular random zonotope
cannot be isotropic. For instance, we need to use a large $N$ in order
to describe a random set built as an isotropic random square; see
Example~\ref{ex:isotropicsquare}. This is the reason for using a larger
class of random zonotopes.

\begin{definition}[Regular random zonotopes]
For an integer $N>1$, any random compact set taking its values almost
surely in $\mathcal{C}^{(N)}_\infty$ is called a \textit{regular random
zonotope} and can be expressed as
\begin{align*}
X=R_x\Biggl(\bigoplus_{i=1}^N
\alpha_i S_{\theta_i}\Biggr),
\end{align*}
where $x$ is a random variable on $[0,\pi]$, and $\alpha$ is a random
vector taking values in $\mathbb{R}_+^N$. The random vector $\alpha$ is
called a \textit{random face length vector of $X$.}
\end{definition}
\begin{proposition}[Isotropic regular random zonotope]
Let $ X=R_x(\bigoplus_{i=1}^N\alpha_i S_{\theta_i})$ be an isotropic
regular random zonotope. Then $X$ has the same distribution of the
following random set:
\begin{equation}
X\stackrel{a.s.} {=}R_\eta\Biggl(\bigoplus
_{i=1}^N\alpha_i S_{\theta_i}
\Biggr), \label{eq:representationIRRZ}
\end{equation}
where $\eta$ is a uniform random variable on $[0,\pi]$ independent of
$\alpha$.
\end{proposition}

\begin{proof}
Let $ X=R_x(\bigoplus_{i=1}^N\alpha_i S_{\theta_i})$ be an isotropic
regular random zonotope, and $\eta'$ be a uniform random variable
independent of $\alpha$. Because of the isotropy of $X$, the random set
$R_{\eta'}(X)$ has the same distribution as $X$. Let $\eta=x+\eta'[\pi]
$. Then the random set $R_{\eta'}(X)$ can be expressed as $R_{\eta
}(\bigoplus_{i=1}^N\alpha_i S_{\theta_i})$. Consequently, $R_{\eta
}(\bigoplus_{i=1}^N\alpha_i S_{\theta_i})$ has the same distribution as $X$.

Let us show that $\eta$ is a uniform variable independent of $\alpha$.

Let $B$ be a Borel set of $\mathbb{R}^N$, and let $E=\lbrace\eta\in
[0,t]\rbrace\cap\lbrace\alpha\in B\rbrace$ for any $t\in[0,\pi]$. Then
\begin{align*}
E &= \lbrace\alpha\in B\rbrace\cap\biggl(\bigcup_{z\in[0,\pi]}
\lbrace x=z\rbrace\cap\bigl\lbrace\eta'+z[\pi]\leq t\bigr\rbrace
\biggr)
\\
&=\bigcup_{z\in[0,\pi]} \lbrace\alpha\in B\rbrace \lbrace
x=z\rbrace\cap \bigl\lbrace\eta'+z[\pi]\leq t\bigr\rbrace.
\end{align*}
Note that this union is disjointed. Then because of the independence of
$\eta'$,
\begin{align*}
\mathbb{P}(E)=\int_0^\pi\mathbb{P}\bigl(
\lbrace\alpha\in B\rbrace \lbrace x=z\rbrace\bigr)\mathbb{P}\bigl( \bigl\lbrace
\eta'+z[\pi]\leq t\bigr\rbrace\bigr)\,dz.
\end{align*}
The quantity $\mathbb{P}( \lbrace\eta'+z[\pi]\leq t\rbrace)$ is
independent of the value of $z$ and can be easily computed as $\mathbb
{P}( \lbrace\eta'+z[\pi]\leq t\rbrace)=\frac{t}{\pi}$. Consequently:
\begin{align*}
\mathbb{P}(E)&=\frac{t}{\pi}\int_0^\pi
\mathbb{P}\bigl( \lbrace\alpha\in B\rbrace \lbrace x=z\rbrace\bigr)\mathbb{P}
\bigl( \bigl\lbrace\eta'+z[\pi]\leq t\bigr\rbrace \bigr)\,dz,
\\
\mathbb{P}(E)&=\frac{t}{\pi}\mathbb{P}\bigl( \lbrace\alpha\in B\rbrace
\bigr).
\end{align*}
Then $\eta$ is a uniform random variable on $[0,\pi]$ independent of
$\alpha$.
\end{proof}

This proposition shows that an isotropic regular random zonotope can
always be described as in \eqref{eq:representationIRRZ}. Such a
zonotope is consequently defined by its random face length vector
$\alpha$. However, different distributions of $\alpha$ can lead to the
same distribution of $X$, as mentioned in the following proposition.

\begin{proposition}[Family of the random face length vectors]
Let $\alpha$ be a random face length vector of the isotropic regular
random zonotope $X$. The following family of random face length
vectors, denoted $\mathcal{F}_{N}(X)$, provides the same distribution
of the random set $X$:
\begin{equation}
\mathcal{F}_{N}(X)=\bigl\lbrace\alpha'\overset{a.s}
{=}J^n\alpha\big\vert\forall \omega\in\varOmega\text{ a.s.}, n(\omega)\in
\lbrace0,\dots,N-1\rbrace \bigr\rbrace,
\end{equation}
where $J$ is the circulant matrix $J=Circ(0,1,0,\dots,0)$.
\end{proposition}

\begin{proof}
First of all, it is easy to see that $\mathcal{F}_{N}(X)$ is not empty
by construction of $X$.
Let $\alpha,\alpha'$ be two representative random vectors of $X$. Then
there exist two uniform random variables $\eta$ and $\eta'$ satisfying
$\eta\independent\alpha$ and $\eta'\independent\alpha'$ such that:
\begin{align*}
&\forall\omega\in\varOmega\text{ a.s.},\quad \bigoplus_{i=1}^N
\alpha_i(\omega ) S_{\theta_i+\eta(\omega)}=\bigoplus
_{i=1}^M \alpha'_i(\omega)
S_{\theta_i+\eta'(\omega)}
\\
\Rightarrow&\quad \forall\omega\in\varOmega\text{ a.s.},\quad  R_{-\eta'(\omega
)}\Biggl(
\bigoplus_{i=1}^N \alpha_i(
\omega) S_{\theta_i+\eta(\omega
)}\Biggr)=R_{-\eta'(\omega)}\Biggl(\bigoplus
_{i=1}^N \alpha'_i(\omega)
S_{\theta
_i+\eta'(\omega)}\Biggr)
\\
\Rightarrow&\quad \forall\omega\in\varOmega\text{ a.s.},\quad \bigoplus
_{i=1}^N \alpha'_i(\omega)
S_{\theta_i}=\bigoplus_{i=1}^N
\alpha_i(\omega) S_{\theta_i+\eta(\omega)-\eta'(\omega)}.
\end{align*}
Then, because of the uniqueness of the face length vector in $\mathcal
{C}_0^{(N)}$, for any $\omega\in\varOmega\text{ a.s., }$ there is $j(\omega
)\in\lbrace1,\dots,N\rbrace$ such that
\begin{align*}
& \theta_1=\bigl(\theta_{j(\omega)}+ \eta(\omega)-
\eta'(\omega)\bigr)[\pi] \text{\quad and\quad } \alpha'_1(
\omega)=\alpha_j(\omega)
\\
\Rightarrow\quad & \theta_{j(\omega)}=\bigl(\eta'(\omega)-\eta(
\omega)\bigr)[\pi]\text{\quad and\quad } \alpha'_1(\omega)=
\alpha_j(\omega)
\\
\Rightarrow\quad &  \alpha'_i(\omega)=\alpha_{i+j-1[M]}(
\omega)
\\
\Rightarrow\quad & \alpha'(\omega)=J^{j(\omega)-1}\alpha(\omega).
\end{align*}
By taking $\forall\omega\in\varOmega\text{ a.s.},\;n(\omega)=j(\omega
)-1[N]$ it follows that $\alpha'=J^n \alpha$ and, consequently,
$\mathcal{F}_{N}(X)\subset\lbrace\alpha'=J^n\alpha\vert\forall\omega\in
\varOmega\text{ a.s.},\ n(\omega)\in\lbrace0,\dots,N-1\rbrace\rbrace$.

The other inclusion can be proved by taking $\eta'$ such that $\forall
\omega\in\varOmega\text{ a.s.}, \eta'(\omega)=\beta_{n(\omega)+1}+\eta[\pi
]$. For such $\eta'$, it follows that $\forall\omega\in\varOmega\text{
a.s.},\;X(\omega)=\bigoplus_{i=1}^N \alpha'_i(\omega) S_{\theta_i+\eta
'(\omega)}$.
\end{proof}

\begin{definition}[Central random face length vector]
Let $\alpha\in\mathcal{F}_{N}(X)$, and let $n$ be a uniform random
variable on $\lbrace0,\dots,M-1\rbrace$ independent of $\alpha$. Then
the random face length vector $\alpha'=J^n\alpha$ is called a \textit
{central random face length vector of $X$}.
\end{definition}
Notice that a central random face length vector has all components
identically distributed. Furthermore, its distribution has many
interesting properties.
\begin{proposition}[Uniqueness of the central face length distribution]
There is a unique distribution for any central random face length
vectors. In other words, let $\tilde{\alpha}',\alpha'$ be two central
random face length vectors of $X$. Then they have same distribution.
Such a distribution will be named \textit{the central face length
distribution of~$X$}.\looseness=1
\label{prop:uniqCentralDisrt}
\end{proposition}

\begin{proof}
Let $\tilde{\alpha}'$ and $\alpha'$ be two central representations of
$X$. Then there exist a random face length vector $\tilde{\alpha}$ and
an independent uniform variable $\tilde{n}$ on $\lbrace0,\dots
,N-1\rbrace$ such that $\tilde{\alpha}'=J^{\tilde{n}}\tilde{\alpha}$.
In addition, $\tilde{\alpha}\in\mathcal{F}_{N}(X)$, so there exists $n$
such that $\tilde{\alpha}=J^{n}\alpha'$. Consequently, $\tilde{\alpha
}'=J^{\tilde{n}+n}\alpha'$. Let $n'=\tilde{n}+n[N]$. It is easy to see
that $ J^{\tilde{n}+n}=J^{n'}$, and thus
\begin{align*}
&\tilde{\alpha}'=J^{n'}\alpha'.
\end{align*}
Let us prove that $n'$ is a uniform variable on $\lbrace0,\dots
,M-1\rbrace$ independent of~$\alpha'$.

For any $k\in\lbrace0,\dots,N-1\rbrace$,
\begin{align*}
\mathbb{P}\bigl(\bigl\lbrace n'=k\bigr\rbrace\bigr)&=\mathbb{P}
\Biggl(\bigcup_{i=0}^{N-1}\bigl\lbrace
\tilde{n}=k-i[N]\bigr\rbrace\cap\lbrace n=i\rbrace\Biggr)
\\
&=\sum_{i=0}^{N-1}\mathbb{P}\bigl(\bigl
\lbrace\tilde{n}=k-i[N]\bigr\rbrace\bigr)\mathbb {P}\bigl(\lbrace n=i\rbrace\bigr)
\\
&=\frac{1}{N}.
\end{align*}
Then $n'$ is a uniform variable on $\lbrace0,\dots,N-1\rbrace$.
Furthermore, for any Borel set $B$ and any $k\in\lbrace0,\dots
,N-1\rbrace$,
\begin{align*}
\mathbb{P}\bigl(\bigl\lbrace n'=k\bigr\rbrace\cap\bigl\lbrace
\alpha'\in B\bigr\rbrace\bigr)&=\mathbb {P}\Biggl(\bigcup
_{i=0}^{N-1}\bigl\lbrace\tilde{n}=k-i[N]\bigr\rbrace
\cap\lbrace n=i\rbrace\cap\bigl\lbrace\alpha'\in B\bigr\rbrace
\Biggr)
\\
&=\sum_{i=0}^{N-1}\mathbb{P}\bigl(\bigl
\lbrace\tilde{n}=k-i[N]\bigr\rbrace\cap\lbrace n=i\rbrace\cap\bigl\lbrace
\alpha'\in B\bigr\rbrace\bigr)
\\
&=\sum_{i=0}^{N-1}\mathbb{P}\bigl(\bigl
\lbrace\tilde{n}=k-i[N]\bigr\rbrace\bigr)\mathbb {P}\bigl(\lbrace n=i\rbrace\cap
\bigl\lbrace\alpha'\in B\bigr\rbrace\bigr)
\\
&=\frac{1}{N}\sum_{i=0}^{N-1}
\mathbb{P}\bigl(\lbrace n=i\rbrace\cap\bigl\lbrace \alpha'\in B\bigr
\rbrace\bigr)
\\
&=\frac{1}{N}\mathbb{P}\bigl(\bigl\lbrace\alpha'\in B\bigr
\rbrace\bigr)
\\
&=\mathbb{P}\bigl(\bigl\lbrace n'=k\bigr\rbrace\bigr)\mathbb{P}
\bigl( \bigl\lbrace\alpha'\in B\bigr\rbrace\bigr).
\end{align*}
Now let us prove that $\alpha'$ and $\tilde{\alpha}'$ have the same
distribution. Let $B=B_{0}\times\cdots\times B_{N-1}$ be a product of
Borel sets of $\mathbb{R}$. Firstly, note that $\mathbb{P}(J^k\alpha
'\in B)= \mathbb{P}(\alpha'\in B)$ for all $k\in\lbrace0,\dots
,N-1\rbrace$. Indeed, by definition, $\alpha'$ can be written as $
\alpha'=J^n\alpha$ with $\alpha$ a representative of $X$ and $n$ an
independent uniform random variable on $\lbrace0,\dots,N-1\rbrace$. Therefore,
\begin{align*}
\mathbb{P}\bigl(\bigl\lbrace\alpha'\in B\bigr\rbrace\bigr)&=
\mathbb{P}\Biggl(\bigcup_{i=0}^{N-1}\bigl
\lbrace J^i\alpha\in B\bigr\rbrace\cap\lbrace n=i\rbrace\Biggr)
\\
&=\frac{1}{N}\sum_{i=0}^{N-1}
\mathbb{P}\bigl(\bigl\lbrace J^i\alpha\in B\bigr\rbrace\bigr)
\\
&=\frac{1}{N}\sum_{i=0}^{N-1}
\mathbb{P}\bigl(\lbrace\alpha\in B_i\times \cdots B_0
\cdots B_{N-1-i}\rbrace\bigr).
\end{align*}
In the same manner,
\begin{align*}
\mathbb{P}\bigl(\bigl\lbrace J^k\alpha'\in B\bigr
\rbrace\bigr)&=\frac{1}{N}\sum_{i=0}^{N-1}
\mathbb{P}\bigl(\bigl\lbrace J^{i+k}\alpha\in B\bigr\rbrace\bigr)
\\
&=\frac{1}{N}\sum_{i=0}^{N-1}
\mathbb{P}\bigl(\lbrace\alpha\in B_{i+k}\times \cdots B_0
\cdots B_{N-1-i-k}\rbrace\bigr)
\\
&=\frac{1}{N}\sum_{i=0}^{N-1}
\mathbb{P}\bigl(\lbrace\alpha\in B_i\times \cdots B_0
\cdots B_{N-1-i}\rbrace\bigr)
\\
&=\mathbb{P}\bigl(\bigl\lbrace\alpha'\in B\bigr\rbrace\bigr).
\end{align*}
Furthermore,
\begin{align*}
\mathbb{P}\bigl(\bigl\lbrace\tilde{\alpha'}\in B\bigr\rbrace
\bigr)&=\mathbb{P}\Biggl(\bigcup_{k=0}^{N-1}
\bigl\lbrace J^k\alpha'\in B\bigr\rbrace\cap\bigl\lbrace
n'=k\bigr\rbrace\Biggr)
\\
&=\frac{1}{N}\sum_{k=0}^{N-1}
\mathbb{P}\bigl(\bigl\lbrace J^k\alpha'\in B\bigr\rbrace
\bigr)
\\
&=\mathbb{P}\bigl(\bigl\lbrace\alpha'\in B\bigr\rbrace\bigr).
\end{align*}
Finally, $\tilde{\alpha}'$ and $\alpha'$ have the same distribution.
\end{proof}

\begin{proposition}[Properties of the central face length distribution]
Let $\alpha$ be a central random face length vector of $X$. Then the
first- and second-order moments of its distribution have the following
properties:
\begin{enumerate}
\item First-order moment:
\begin{equation}
\forall i=1,\dots,N,\quad \mathbb{E}[\alpha_i]=\frac{U(X)}{2N};
\end{equation}
\item Second-order moment:

The matrix $C[\alpha]=(\mathbb{E}[\alpha_i\alpha_j])_{1\leq i,j\leq N}$
is a circulant matrix defined by the first column $V[\alpha]=^t(\mathbb
{E}[\alpha_1\alpha_1],\dots,\mathbb{E}[\alpha_1\alpha_N])$: $C[\alpha]
=Circ(V[\alpha])$. Furthermore, this matrix is symmetric and depends
only on $(\lfloor\frac{N}{2}\rfloor+1)$ values, where $\lfloor\frac
{N}{2}\rfloor$ denotes the floor of $\frac{N}{2}$. Note that $
m=\lfloor\frac{N}{2}\rfloor$ and $v=^t(\mathbb{E}[\alpha_1\alpha
_1],\dots,\mathbb{E}[\alpha_1\alpha_{m+1 }])$; therefore, if $N$ is an
even integer, then $V=^t(v_0,\dots,v_{m-1},v_{m},v_{m-1},\dots,v_1)$,
and if $N$ is an odd integer, then $V=^t(v_0,\dots,v_m,v_{m},\dots,v_1)$.
\end{enumerate}
\label{prop:propCentralDistr}
\end{proposition}
\begin{proof}
$\:$
\begin{enumerate}
\item[1.] The first item is trivial. Indeed, the marginals of $\alpha$
are identically distributed. Therefore, $\mathbb{E}[\alpha_i]=\mathbb
{E}[\alpha_j]$ for all $i,j$, and $U(X)=2\sum_{i=1}^N \alpha
_i\Rightarrow\;\mathbb{E}[\alpha_i]=\frac{U(X)}{2N}$, $i=1,\dots,N$.
\item[2.] It has been shown that if, for any $k\in\lbrace0,\dots
,N-1\rbrace$, the random variables $\alpha$ and $J^k\alpha$ have same
distribution, then they have the same covariance matrix. Therefore, for
all $1\leq i,j\leq N$,
\begin{align*}
\forall k\in\lbrace0,\dots,N-1\rbrace,\quad  \mathbb{E}[\alpha_i\alpha
_j]=\mathbb{E}[\alpha_{i+k[N]+1}\alpha_{j+k[N]+1}],
\end{align*}
so $\mathbb{E}[\alpha_i\alpha_j]$ is a circulant matrix that depends
only on $i-j[N]$ and, because of its symmetry, and on $ j-i[N]$. Let
$1\leq i\leq j\leq N$. Then there are two possible cases. First,
suppose that $N=2m$ is an even integer. Then, for all $ 0\leq k\leq m-1$,
\begin{align*}
\mathbb{E}[\alpha_1\alpha_{1+m+k}]=\mathbb{E}[
\alpha_{1+m}\alpha _{1+k}]=\mathbb{E}[\alpha_{1+m+N-k}
\alpha_{1+N}]=\mathbb{E}[\alpha _{1}\alpha_{1+m-k}].
\end{align*}
Note that
\[
V=^t\bigl(\mathbb{E}[\alpha_1\alpha_1],
\dots,\mathbb{E}[\alpha_1\alpha_N]\bigr) \ \mbox{ and }\
v=^t\bigl(\mathbb{E}[\alpha_1\alpha_1],
\dots,\mathbb {E}[\alpha_1\alpha_{m+1 }]\bigr).
\]
Therefore, there is
\[
V=^t(v_0,\dots,v_{m-1},v_{m},v_{m-1},
\dots,v_1).
\]
If $N$ is an odd integer, then $N=2m+1$, and for any $0\leq k\leq m$,
\begin{align*}
\mathbb{E}[\alpha_1\alpha_{1+m+k}]=\mathbb{E}[
\alpha_{1+m+1}\alpha _{1+k}]=\mathbb{E}[\alpha_{2+m+N-k}
\alpha_{1+N}]=\mathbb{E}[\alpha _{1}\alpha_{2+m-k}],
\end{align*}
then by noting that
\[
V[\alpha]=^t\bigl(\mathbb{E}[\alpha_1
\alpha_1],\dots,\mathbb{E}[\alpha _1
\alpha_N]\bigr) \ \mbox{ and }\ v=^t\bigl(\mathbb{E}[
\alpha_1\alpha_1],\dots, \mathbb{E}[
\alpha_1\alpha_{m+1 }]\bigr)
\]
there is $V=^t(v_0,\ldots v_m,v_{m},\ldots v_1)$. Finally, $C[\alpha]$
is a symmetric circulant matrix.\qedhere
\end{enumerate}
\end{proof}

\begin{example}
In order to illustrate the properties of the face length vector
distributions, let us discuss the case $N=2$. Then, $X=R_\eta(\alpha_1
S_0\oplus\alpha_2 S_{\frac{\pi}{2}})$ with $\eta$ a uniform random
variable on $[0,\pi]$ independent of $\alpha$.

Therefore, $X$ is an isotropic random rectangle described by its sides
$(\alpha_1,\alpha_2)$. However, this is not the unique way to describe
it. Indeed, even for a deterministic rectangle of sides $(a,b)$, we can
also say that its sides are $(b,a)$. This simple fact involves a lot of
different distributions for the face length vectors of an isotropic
random rectangle.

Let us take consider a simple example: suppose that $Y$ is equiprobably
the rectangle of sides $(1,2)$ or the rectangle of sides $(3,4)$. Then,
there is at least one of the following four possible descriptions for
the realization of sides of $Y$:
\begin{list}{$\bullet$}{}
\item$(1,2)$ or $(3,4)$;
\item$(2,1)$ or $(3,4)$;
\item$(2,1)$ or $(4,3)$;
\item$(1,2)$ or $(4,3)$.
\end{list}
Therefore, there are four corresponding face length distributions $\frac
{1}{2}\Delta_{(1,2)}+\frac{1}{2}\Delta_{(3,4)}$,
$\frac{1}{2}\Delta_{(2,1)}+\frac{1}{2}\Delta_{(3,4)}$,\ldots, where $
\Delta_{(a,b)}$ denotes the Dirac measure in $(a,b)$. However, there
are not the only possibilities. Indeed, many other can be built from
the previous distributions, such as the distribution $\frac{1}{4}\Delta
_{(1,2)}+\frac{1}{4}\Delta_{(2,1)}+\frac{1}{2}\Delta_{(3,4)}$. Notice
that the central distribution of $Y$ is $\frac{1}{4}\Delta_{(1,2)}+\frac
{1}{4}\Delta_{(2,1)}+\frac{1}{4}\Delta_{(3,4)}+\frac{1}{4}\Delta
_{(4,3)}$.

Let us return now to the general case of the isotropic random rectangle
$X$ with a face length vector $\alpha$. According to the foregoing, it
is easy to see that any another face length vector $\alpha'$ of $X$ can
be built as
\begin{equation}
\alpha'= %
\begin{pmatrix}
1-\delta& \delta\\
\delta&1-\delta
\end{pmatrix} %
\alpha, \label{eq:example:RandomRectange}
\end{equation}
where $\delta$ is any Bernoulli variable \textup{(}i.e., valued in
$\lbrace0,1\rbrace)$ eventually correlated to $\alpha$.

Indeed, notice that $
\xch{(\begin{smallmatrix}
1-\delta& \delta\\
\delta&1-\delta
\end{smallmatrix})}{\begin{pmatrix}
1-\delta& \delta\\
\delta&1-\delta
\end{pmatrix}}
=J^\delta$, and therefore by taking $\eta'=\eta+\delta\frac{\pi}{2}[\pi
]$ and
\begin{align*}
X=R_{\eta}(\alpha_1 S_0\oplus
\alpha_2 S_\frac{\pi} {2})=R_{\eta'}\bigl(\alpha
'_1 S_0\oplus\alpha'_2
S_\frac{\pi} {2}\bigr),
\end{align*}
we can easily prove that $\eta'$ is a uniform random variable on $[0,\pi
]$ independent of $\alpha'$ \textup{(}see the proof of Prop.~\ref
{prop:uniqCentralDisrt}), and therefore $\alpha'$ is a face length
vector of $X$.

Let us consider now the central face length distribution, so let $\delta
$ be a~Bernoulli variable of parameter $\frac{1}{2}$ \textup{(}i.e., a
uniform variable on $\lbrace0,1\rbrace)$ independent of $\alpha$, and
let $\alpha'=J^\delta\alpha$ be a central face length vector. Then,
according to \eqref{eq:example:RandomRectange},
\begin{align*}
\alpha'_1 &= (1-\delta)\alpha_1 +\delta
\alpha_2,
\\
\alpha'_2 &= \delta\alpha_1 +(1-\delta)
\alpha_2.
\end{align*}

Consequently, the first- and second-order moments of the face length
distribution can be computed as
\begin{align*}
&\mathbb{E}\bigl[\alpha'_1\bigr] =\mathbb{E}\bigl[
\alpha'_2\bigr]=\frac{1}{2}\mathbb {E}[
\alpha_1+\alpha_2],
\\
&\mathbb{E}\bigl[{\alpha'_1}^2\bigr] =
\mathbb{E}\bigl[{\alpha'_2} ^2\bigr]=
\frac
{1}{2}\mathbb{E}\bigl[\alpha_1^2+{
\alpha_2}^2\bigr],
\\
&\mathbb{E}\bigl[\alpha'_1\alpha'_2
\bigr]=\mathbb{E}[\alpha_1\alpha_2].
\end{align*}
Notice that property~\ref{prop:propCentralDistr} is well verified,.
Indeed, $\mathbb{E}[\alpha'_1] =\mathbb{E}[\alpha'_2]=\frac{1}{4}\mathbb
{E}[U(X)]$, and the matrix $C[\alpha]$ is a circulant matrix depending
on two parameters.
\end{example}

\subsection{Characterizing an isotropic regular random zonotope from
its Feret diameter random process}

We have shown that the distribution of an isotropic random zonotope
$X$ can be described by its central face length distribution and
studied the properties of such distributions. Here we will show how its
characteristics can be connected to the geometrical characteristics of
the random zonotope. In particular, we will give formulae that allow us
to connect the first- and second-order moments of the Feret diameter of
$X$ to those of the central face length distribution.

%\label{sec:facelengthmoment}
Let $X$ be an isotropic random zonotope represented by its face length
vector~$\alpha$. Let us recall that $X$ can be almost surely expressed as
\begin{equation*}
X=R_{\eta}\Biggl(\bigoplus_{i=1}^N
\alpha_i S_{\theta_i}\Biggr),
\end{equation*}
where $\eta$ is a uniform random variable independent of $\alpha\geq
0$. Suppose that the condition $\mathbb{E}(U(X)^2)<\infty$ is
satisfied. Then, according to Proposition~\ref{prop:CondExistMoments},
$\alpha\in L^2(\mathbb{R}_+^N)$, and the mean and autocovariance of
$H_{X}$ exist.

According to Proposition~\ref{prop:A P H deterministe}, for any
representative $\alpha$ of $X$, some geometrical characteristics of $X$
can be expressed as
\begin{align*}
\forall t\in\mathbb{R},\; H_{X}(t)&=\sum_{i=1}^N\alpha_i\bigl\vert\sin (t-\eta-\theta_i )\bigr\vert,\\
U(X)&=2\sum_{i=1}^N \alpha_i,\\
A(X)&=\frac{1}{2}\sum_{i=1}^N \sum_{j=1}^N\alpha_i\alpha_j\bigl\vert\sin (\theta_i-\theta_j)\bigr\vert.
\end{align*}
Therefore, by considering $\alpha\in L^2(\mathbb{R}_+^N)$ and the
independence of $\alpha$ and $\eta$, their expectation can be computed
by integration with respect to the uniform distribution of~$\eta$:
\begin{align}
\forall t\in\mathbb{R},\quad \mathbb{E}\bigl[ H_{X}(t)\bigr]&=\frac{2}{\pi}\sum_{i=1}^N \mathbb{E}[\alpha_i], \label{eq:espFeretModel}\\
\mathbb{E}\bigl[ U(X)\bigr]&=2\sum_{i=1}^N\mathbb{E}[ \alpha_i],\\
A(X)&=\frac{1}{2}\sum_{i=1}^N \sum_{j=1}^N\mathbb{E}[ \alpha_i\alpha _j]\bigl\vert\sin(\theta_i-\theta_j)\bigr\vert,\label{eq:espAreaIsotropicregZonotope}\\
\forall t,t'\in\mathbb{R},\quad \mathbb{E}\bigl[ H_{X}(t)H_{X}\bigl(t+t'\bigr)\bigr]&=\sum_{i=1}^N\sum_{j=1}^N\mathbb{E}[ \alpha_i\alpha_j]k_S\bigl(t'+\theta_i -\theta_j\bigr),\label{eq:HHIsotropicregZonotope}\\
\text{where } \forall t\in\mathbb{R},k_S(t)&=\frac{1}{\pi}\int_0^\pi \bigl\vert\sin(t+z)\sin(z)\bigr\vert \,dz.
\end{align}
Note that $k_S$ is a $\pi$-periodic function and can be expressed on
$[0,\pi]$ as
\begin{equation}
k_S(t)=\frac{1}{2\pi}\bigl(2\sin^3(t)+\cos(t)
\bigl(\pi-2t+\sin(2t)\bigr)\bigr). \label{eq:exppression_k_s}
\end{equation}

Using Eq.~\eqref{eq:HHIsotropicregZonotope} and the stationarity of
$H_{X}$, we have
\begin{align}
&\forall t,t'\in\mathbb{R},\quad \mathbb{E}\bigl[ H_{X}(t)
H_{X}\bigl(t+t'\bigr)\bigr]=\mathbb {E}\bigl[
H_{X}(t) H_{X}\bigl(t-t'\bigr)\bigr],
\\
&\forall t,t'\in\mathbb{R},\quad  \sum_{i=1}^N
\sum_{j=1}^N \mathbb{E}[
\alpha_i\alpha_j]k_S\bigl(t'+
\theta_i -\theta_j\bigr)=\sum
_{j=1}^N \sum_{i=1}^N
\mathbb{E}[ \alpha_i\alpha_j]k_S
\bigl(t'+\theta_j -\theta_i\bigr),
\end{align}
and by introducing the functional
\begin{equation}
\forall t\in\mathbb{R},\forall1\leq i,j\leq N,\quad  K_{ij}(t)=k_S(t+
\theta _i -\theta_j)
\end{equation}
it follows that
\begin{equation}
\forall t,t'\in\mathbb{R},\quad \mathbb{E}\bigl[ H_{X}(t)
H_{X}\bigl(t+t'\bigr)\bigr]=\sum
_{i=1}^N\sum_{j=1}^N
\mathbb{E}[ \alpha_i\alpha_j]K_{ij}
\bigl(t'\bigr).
\end{equation}
\begin{proposition}
For any real $t$, $K(t)$ is a circulant matrix. Furthermore, by
denoting $((k_1(t),\dots,k_{N}(t)))$ the first line of $K(t)$, we have
$K(t)=\operatorname{Circ}((k_1(t),\dots,k_{N}(t)))$ and $K_{ij}(t)=k_j(\theta_i+t)$
for $1\leq i,j\leq N$.
\label{prop:propmatriceKt}
\end{proposition}

\begin{proof}
Let us show that $K(t)$ is a circulant matrix. For any real $t$,
$t+\theta_i -\theta_j$ depends only on $ i-j$; therefore, $K(t)$ is a
Toeplitz matrix. Furthermore, for $1\leq i\leq N-1$ and $1\leq j\leq
N,\;K_{(i+1)j}(t)=k_S(t+\theta_i-(\theta_j-\frac{\pi}{N}))$, but
$(\theta_j-\frac{\pi}{N})=\theta_{\sigma(j)}$ where $\sigma
(j)=(j-2[N])+1$, and thus $K_{(i+1)j}(t)=K_{i\sigma(j)}$. Therefore,
the line index $i+1$ of $K(t)$ is a cyclic permutation of the line
index $i$ of $K(t)$, so $K(t)$ is a circulant matrix. Furthermore,
$k_{j}(\theta_{i}+t)=k_{S}(t+\theta_i-\theta_j)=K_{ij}(t)$.
\end{proof}

Suppose now that $\alpha$ is a central representative of $X$. We will
show that the first- and second-order moments of the central
distribution can be easily expressed from the Feret diameter process.

\begin{theorem}[Moments of the central face length distribution]
Let $X$ be an isotropic random zonotope represented by a central face
length vector $\alpha$. Then
\begin{equation}
\forall x\in\mathbb{R}\quad \forall i=1,\dots,N,\quad  \mathbb{E}[\alpha _i]=\frac{\pi}{2N}\mathbb{E}\bigl[H_{X}(x)\bigr],
\end{equation}
\begin{equation}
V[\alpha]=\frac{1}{N}K(0)^{-1} V\bigl[H^{(N)}_{X}
\bigr], \label{eq:systemlineaireHHtoalpha}
\end{equation}
where $V[x]$ denotes the vector $^t(\mathbb{E}[x_1 x_1],\dots,\mathbb
{E}[x_1 x_N])$.
\end{theorem}

\begin{proof}
Suppose that $\alpha$ is a central representative of $X$. Then,
according to Proposition~\ref{prop:propCentralDistr} and Eq.~\eqref
{eq:espFeretModel}, the first-order moment of the central distribution
can be expressed as
\begin{equation}
\forall x\in\mathbb{R}\quad \forall i=1,\dots, N,\quad  \mathbb{E}[\alpha
_i]=\frac{\pi}{2N}\mathbb{E}\bigl[H_{X}(x)\bigr].
\end{equation}
By Propositions~\ref{prop:propmatriceKt} and \ref
{prop:propCentralDistr}, it follows that $\mathbb{E}[\alpha_i\alpha
_j]=V[\alpha]_{j-i[N]+1}$ and $K_{ij}(t)=k_{j-i[N]+1}$. Then, for all
$t\in\mathbb{R}$,
\begin{align}
&\mathbb{E}\bigl[ H_{X}(0) H_{X}(t)\bigr]=\sum_{i=1}^N\sum_{j=1}^N\mathbb{E}[ \alpha_i\alpha_j]K_{ij}(t)\notag\\
&\quad =\sum_{i=1}^N\sum_{j=1}^N V[\alpha]_{j-i[N]+1}k_{j-i[N]+1}(t)\notag\\
&\quad =\sum_{i=1}^N\sum_{j=1}^i V[\alpha]_{j-i[N]+1}k_{j-i[N]+1}(t)+\sum_{i=1}^N\sum_{j=i}^N V[\alpha]_{j-i[N]+1}k_{j-i[N]+1}(t)\notag\\
&\qquad -\sum_{i=1}^N V[\alpha]_{1}k_{1}(t)\notag\\
&\quad =\sum_{i=1}^N\sum_{s=0}^{i-1} V[\alpha]_{s+1}k_{s +1}(t)+\sum_{i=1}^N\sum_{s=0}^{N-i} V[\alpha]_{s+1}k_{N-s[N]+1}(t)- N V[\alpha ]_{1} s_{1}(t)\notag\\
&\quad =\sum_{i=1}^N\sum_{s=0}^{i-1} V[\alpha]_{s+1}k_{s +1}(t)+\sum_{i=1}^N\sum_{z=N}^{i} V[\alpha]_{N-z+1}k_{z[N]+1}(t)- N V[\alpha]_{1} k_{1}(t)\notag\\
&\quad =\sum_{i=1}^N\sum_{s=0}^{i-1} V[\alpha]_{s+1}k_{s +1}(t)+\sum_{i=1}^N\sum_{z=N}^{i} V[\alpha]_{z[N]+1}k_{z[N]+1}(t)- N V[\alpha ]_{1} k_{1}(t)\notag\\
&\quad =\sum_{i=1}^N\sum_{s=0}^{i-1} V[\alpha]_{s+1}k_{s +1}(t)+\sum_{i=1}^N\sum_{s=i}^{N} V[\alpha]_{s[N]+1}k_{s[N]+1}(t)- N V[\alpha ]_{1} k_{1}(t)\notag\\
&\quad =\sum_{i=1}^N\sum_{s=0}^{i-1} V[\alpha]_{s+1}k_{s +1}(t)+\sum_{i=1}^N\sum_{s=i}^{N} V[\alpha]_{s[N]+1}k_{s[N]+1}(t)- N V[\alpha ]_{1} k_{1}(t)\notag\\
&\quad =\sum_{i=1}^N\sum_{s=0}^{N} V[\alpha]_{s[N]+1}k_{s[N]+1}(t)- N V[\alpha]_{1} k_{1}(t)\notag\\
&\quad =\sum_{i=1}^N\sum_{s=0}^{N-1} V[\alpha]_{s[N]+1}k_{s[N]+1}(t)+\sum_{i=1}^N V[\alpha]_{1} k_{1}(t) - N V[\alpha]_{1}k_{1}(t)\notag\\
&\quad =\sum_{i=1}^N\sum_{s=0}^{N-1} V[\alpha]_{s+1}k_{s+1}(t)\notag\\
&\quad =\sum_{i=1}^N\sum_{s=1}^{N} V[\alpha]_{s}k_{s}(t)\notag\\
\Rightarrow&\quad \mathbb{E}\bigl[ H_{X}(0) H_{X}(t)\bigr]=N\sum_{s=1}^{N} V[\alpha ]_{s}k_{s}(t). \label{eq:lienHHtoalpha}
\end{align}
Note that $V[H_X^{(N)}]=^t(\mathbb{E}[ H_{X}(0) H_{X}(\theta_1)],\ldots
\mathbb{E}[ H_{X}(0) H_{X}(\theta_N)])$. Since for $1\leq i\leq N$,
$V[H_X^{(N)}]_i= N\sum_{s=1}^{N} V[\alpha]_{s} k_{s}(\theta_i)=N\sum_{s=1}^{N} V[\alpha]_{s} K_{is}(0)$, we have that
\begin{equation}
V\bigl[H_X^{(N)}\bigr]=N K(0)V[\alpha]. \label{eq: matricielCHetV}
\end{equation}
It is easy to see that $K(0)$ is a symmetric positive definite matrix. Indeed,
for $1\leq i,j\leq N,\; K_{ij}(0)=k_S(\theta_i-\theta_j)=k_S(\theta
_j-\theta_i)=K_{ji}(0)$, and \xch{this}{thys} $K(0)$ is a symmetric matrix.
Furthermore, for all $x\in\mathbb{R}^N$,
\begin{align*}
^t xK(0)x &=\sum_{i=1}^N\sum
_{j=1}^N x_i x_j
K_{ij}(0)
\\
&=\sum_{i=1}^N\sum
_{j=1}^N x_i x_j
\frac{1}{\pi}\int_0^\pi\bigl\vert\sin (
\theta_i-\theta_j+z)\sin(z)\bigr\vert \,dz
\\
&= \frac{1}{\pi}\int_0^\pi\sum
_{i=1}^N\sum_{j=1}^N
x_i x_j \bigl\vert\sin (z-\theta_i)\sin(z-
\theta_j)\bigr\vert \,dz
\\
&= \frac{1}{\pi}\int_0^\pi\Biggl(\sum
_{i=1}^N x_i\bigl\vert\sin(z-
\theta_i)\bigr\vert \Biggr)^2 \,dz.
\end{align*}
Denote by $Y$ the real-valued random variable
$Y=\sum_{i=1}^N x_i\vert\sin(z-\theta_i)\vert$, where $z$ is a uniform
random variable on $[0,\pi]$. Then $ ^t xK(0)x =\mathbb{E}[Y^2]$, and
so $ ^t xK(0)x \geq0$. Furthermore, $ ^t xK(0)x =0$ if and only if
$Y=0$ almost surely, $Y=0\; \text{a.s.}\Rightarrow\forall z\in[0,\pi
]$, and $\sum_{i=1}^N x_i\vert\sin(z-\theta_i)\vert=0\Rightarrow x=0
$. Finally, $K(0)$ is a symmetric positive definite matrix. Then it is
invertible, and it follows that
\begin{equation*}
V[\alpha]=\frac{1}{N}K(0)^{-1} V\bigl[H_X^{(N)}
\bigr].\qedhere
\end{equation*}
\end{proof}

This theorem gives the first- and second-order moments of the central
face length distribution from those of the Feret diameter.
Note that $V[\alpha]$ and $V[H^{(N)}_{X}]$ satisfy the properties of
symmetry. Indeed, by denoting $m=\lfloor\frac{N}{2}\rfloor$ we have
shown that if $N$ is an even integer, then $V[\alpha]=^t(v_0,\dots
,v_{m-1},v_{m},v_{m-1},\dots,v_1)$ and if $N$ is an odd integer, then
$V[\alpha]=^t(v_0,\dots,v_m,v_{m},\dots, v_1)$, where $v_k=\mathbb
{E}[\alpha_1\alpha_{1+k}],\; k=0,\dots,m$. The vector $V[H^{(N)}_{X}]$
can be expressed in the same way: for $i=1,\dots,N$, $\pi-\theta
_i=\frac{N-i+2-1}{N}\pi=\theta_{N-i+1[N]+1}$, and
\begin{align*}
V\bigl[H^{(N)}_{X}\bigr]_{i}&=\mathbb{E}
\bigl[H_{X}(0)H_{X}(\theta_{i})\bigr]
\\
&=\mathbb{E}\bigl[H_{X}(0)H_{X}(\pi-\theta_{i})
\bigr]
\\
&=V\bigl[H^{(N)}_{X}\bigr]_{N-i+1[N]+1}.
\end{align*}
Therefore, if $N$ is an even integer, then $V[H^{(N)}_{X}]=^t(c_0,\dots,c_{m-1},c_{m},\break c_{m-1},\dots,v_1)$, and if $N$ is an odd integer, then
$V[H^{(N)}_{X}]=^t(c_0,\dots,c_m,c_{m},\dots,c_1)$, where
$c_k=V[H^{(N)}_{X}]_{k+1},\; k=0,\dots,m$. In practice, the vector
$V[\alpha]$ can be computed by the knowledge of the $m+1$ first
components of $V[H^{(N)}_{X}]$, and the linear problem \eqref{eq:
matricielCHetV} can be rewritten and solved as a linear problem of size $m+1$.
\begin{remark}
In practice, the estimation of $\mathbb{E}[ H_{X}(0) H_{X}(t)]$ for $
t\in[0,\pi]$ is often noised. Then, a better choice it is to find
$V[\alpha]$ in the least squares sense. Let $N'\geq m+1$, and let $0=
t_1\leq\cdots\leq t_{N'}=\frac{\pi}{2} $ be a subdivision of $[0,\pi
]$ containing $\lbrace\theta_1,\dots,\theta_{m+1}\rbrace$, the $(t_i)_
{1\leq i\leq{N'}}$ are observation points. Let us recall that for all
$t\in[0,\frac{\pi}{2}],\;$ $\mathbb{E}[ H_{X}(0) H_{X}(t)]=\mathbb{E}[
H_{X}(0) H_{X}(\pi-t)]$. Then we can suppose that there exist $2(N'-1)$
points of observation such that $z_i=t_i $ for $i=1,\dots,N'$ and
$z_i=t_{2N'-i} $ for $i=N'+1,\dots,2N'-2$. Let $ Q_{ij}=k_j(z_i)$ and
$V[H^{(2(N'-1))}_{X}]=^t(\mathbb{E}[ H_{X}(0) H_{X}(z_1)],\dots,\mathbb
{E}[ H_{X}(0) H_{X}(z_{2(N'-1))}])$. Then, by \eqref{eq:lienHHtoalpha},
\begin{align*}
V\bigl[H^{(2(N'-1))}_{X}\bigr]=QV[\alpha].
\end{align*}\eject
\noindent Finally, if $ \hat{V}[H^{(2(N'-1))}_{X}]$ is a noisy estimation of
$V[H^{(2(N'-1))}_{X}]$, then the following least square estimator of
$V[\alpha]$ is better than that provided by \eqref{eq:systemlineaireHHtoalpha}:
\begin{equation}
\tilde{V}[\alpha]=\argmin_{V\in\mathbb{R}^N_+}\bigl\Vert\hat {V}\bigl[H^{(2(N'-1))}_{X}
\bigr]-QV\bigr\Vert^2.
\end{equation}
\end{remark}

We have discussed some properties of the random zonotopes. The
$0$-regular random zonotopes and the regular random zonotope were
defined and studied. We have shown that a $0$-regular random zonotope
can be describes by a unique face length distribution. Such a
distribution can be easily related to the Feret diameter of the
$0$-regular random zonotope by the relations established in
Section~1.

We have studied different face length distributions of a regular random
zonotope. We have shown that, among them, one can be identified, the
central face length distribution. Finally, we have given some formulae
that allow us to compute the first- and second-order moments of the
central face length distribution from those of the Feret diameter of
the regular random zonotope. The following section is devoted to a
description of a random symmetric convex set as a $0$-regular random
zonotope and as a regular random zonotope.
\section{Description of a random symmetric convex set as a random
zonotope from its Feret diameter}

In Section~1, we have defined some approximations of a symmetric convex
set as zonotopes. In Section~2, we characterized the regular and
$0$-regular random zonotopes from their Feret diameters random process.
The aim of this section is to generalize the previous approximations to
a random symmetric convex set~$X$.

Firstly, we will show that the $0$-regular random zonotope
corresponding to the $\mathcal{C}^{(N)}_0$-approximation of $X$ can be
characterized from the Feret diameter random process of $X$. Secondly,
we will show that the isotropic regular random zonotope corresponding
to the $\mathcal{C}^{(N)}_\infty$-approximation of an isotropized set
of $X$ can be estimated from the Feret diameter random process of
$X$.

\subsection{Approximation of a random symmetric convex set by a
$0$-regular random zonotope}
\label{sec:approx0reg}

Here we investigate the approximation of a random symmetric convex set
$X$ by a $0$-regular random zonotope. We show that the random set
$X_0^{(N)}$ valued in $\mathcal{C}_0^{(N)}$ that is defined as the
$\mathcal{C}_0^{(N)}$-approximation of realizations of $X$ can be
characterized from the Feret diameter of $X$. Finally, we give some
formulas that allow us to compute the moments of the random vector of
the faces of~$X_0^{(N)}$.

\begin{proposition}[Approximation by a $0$-regular random zonotope]
Let $X$ be a random convex set.
For any $\omega\in\varOmega\text{ a.s.}$, let $X_0^{(N)}(\omega)$ be the
$\mathcal{C}^{(N)}_0$-approximation of $X(\omega)$. The $0$-regular
random zonotope $X_0^{(N)}$ is called the $\mathcal
{C}^{(N)}_0$-approximation of the random set $X$.\vadjust{\eject}

For any $N>1$, an interval of confidence for the Hausdorff distance can
be built. Indeed, for any $a>0$, we have the relation
\begin{equation}
\mathbb{P}\bigl(d_H\bigl(X,X_0^{(N)}\bigr)>a
\bigr)\leq\frac{(6+2\sqrt{2})}{a}\sin\biggl(\frac{\pi
}{2N}\biggr) \mathbb{E}
\bigl[\operatorname{diam}(X)\bigr]. \label{eq:prop:IntervalConfidenceN-0}
\end{equation}
If $\epsilon\in[0,1]$ is a confidence level, then $a(\epsilon,N)=\frac
{(6+2\sqrt{2})}{\epsilon}\sin(\frac{\pi}{2N}) \mathbb{E}[\operatorname{diam}(X)]$ can
be considered as an upper bound for $d_H(X,X_0^{(N)})$ with confidence
$1-\epsilon$.

Consequently, such an approximation is consistent as $ N\rightarrow
\infty$.
\end{proposition}

\begin{proof}
Let $X$ be a random convex set. For any $\omega\in\varOmega\text{ a.s.}$,
let $X_0^{(N)}(\omega)$ be the $\mathcal{C}^{(N)}_0$-approximation of
$X(\omega)$ in $\mathcal{C}^{(N)}_0$. According to Theorem~\ref
{thm:approx}, for any real $a>0$,
\begin{align*}
&\forall\omega\in\varOmega\text{ a.s.},\quad  d_H\bigl(X,X_0^{(N)}
\bigr)\leq(6+2\sqrt {2})\sin\biggl(\frac{\pi}{2N}\biggr) \operatorname{diam}(X)
\\
\Rightarrow\quad & \mathbb{P}\bigl(d_H\bigl(X,X_0^{(N)}
\bigr)>a\bigr)\leq\mathbb{P}\biggl((6+2\sqrt {2})\sin\biggl(\frac{\pi}{2N}
\biggr) \operatorname{diam}(X)>a\biggr).
\end{align*}
By using the Markov inequality \cite{cox1977theoryRandomProcess} it
follows that
\begin{align*}
\mathbb{P}\bigl(d_H\bigl(X,X_0^{(N)}\bigr)>a
\bigr)\leq\frac{(6+2\sqrt{2})}{a}\sin\biggl(\frac{\pi
}{2N}\biggr) \mathbb{E}
\bigl[\operatorname{diam}(X)\bigr].
\end{align*}
The consistence of the approximation as $N\rightarrow\infty$ follows
directly from this relation.
\end{proof}

According to relation \eqref{eq:equivalenceDiamPer}, $\mathbb
{E}[\operatorname{diam}(X)]$ can be replaced by $\frac{1}{2}\mathbb{E}[U(X)]$.

Let $X$ be a random symmetric convex set, and $X_0^{(N)}$ be its
$\mathcal{C}^{(N)}_0$-approximation. The, in the same way as in the
deterministic case, the face length distribution can be related to the
Feret diameter of $X$.

\begin{proposition}[Characterization of the $\mathcal
{C}^{(N)}_0$-approximation from the Feret diameter process]
Let $N>1$ be an integer, and $X$ be a random symmetric convex set. Let
$X_0^{(N)}$ be the $\mathcal{C}^{(N)}_0$-approximation of $X$. Its face
length vector $\alpha$ can be characterized from the Feret diameter process:
\begin{equation}
\forall\omega\in\varOmega\text{ a.s.}\quad\alpha(\omega)={F^{(N)}}^{-1}
H_X^{(N)}(\omega), \label{eq:lienapprox0reg}
\end{equation}
\begin{equation}
\mathbb{E}[\alpha]={F^{(N)}}^{-1}\mathbb{E}
\bigl[H_X^{(N)}\bigr], \label{eq:Espapprox0reg}
\end{equation}
\begin{equation}
C[\alpha]={F^{(N)}}^{-1}C\bigl[H_X^{(N)}
\bigr]{^t{F^{(N)}}^{-1}}, \label{eq:Covapprox0reg}
\end{equation}
where $ H_X^{(N)}= {^t(}H_X(\theta_1),\dots,H_X(\theta_1))$ is the
random vector composed by the Feret diameter evaluated on the regular
subdivision. The matrix $F^{(N)}$ is still defined as $(\vert\sin
(\theta_i-\theta_j)\vert)_{ij})_{1\leq i,j\leq N}$, and for a vector
$x$, $C[x]$ denotes its second-order moments $\mathbb{E}[x {^tx}]$.
\end{proposition}

\begin{proof}
According to Theorem~\ref{thm:approx}, the matrix $F^{(N)}$ is
invertible, and thus by the definition of the approximation relation
\eqref{eq:lienapprox0reg} follows. Noting that $\alpha{^t\alpha
}={F^{(N)}}^{-1}H_X^{(N)}{^t{H_X^{(N)}}} {^t{F^{(N)}}^{-1}}$, relations
\eqref{eq:Espapprox0reg} and \eqref{eq:Covapprox0reg} follow from the
linearity of the expectation.
\end{proof}

\begin{remark}
The $\mathcal{C}^{(N)}_0$-approximation of a random symmetric convex
set $X$ is a consistent approximation as $N\rightarrow\infty$.
Furthermore, if $X$ is already a $0$-regular random zonotope in
$\mathcal{C}^{(N)}_0$, then its $M$th approximation $X^{(M)}_0$
coincides with $X$ if and only if $N$ is a divider of $M$.

Such an approximation is sensitive to a rotation of $X$. Indeed, if
$R_\eta(X)$ is the rotation of $X$ by the random angle $\eta$, then the
random sets $X$ and $R_\eta(X)$ have different approximations. This
property can be seen as an advantage or a disadvantage. Indeed, if the
objective is to describe the direction of some random set, then it is
an advantage, but there is a need to use large $N$. However, when the
objective is to describe the shape of a random set with a small $N$
without taking into consideration its direction, then it can be a great
disadvantage; see the following example.
\end{remark}

\begin{example}
Let $N=2$, and let $\theta_1=0,\;\theta_2=\frac{\pi}{2}$, the regular
subdivision. Let us consider the random symmetric convex set $X$ as a
deterministic square of side 1, that is, $X=S_{\theta_1}\oplus S_{\theta
_2}$. Its $\mathcal{C}^{(2)}_0$-approximation coincides with $X$:
$X^{(2)}_0=X$. The matrix $F^{(N)}$ is defined as
\begin{align*}
F^{(N)}={F^{(N)}}^{-1}= %
\begin{pmatrix}0&1\\
1&0\\
\end{pmatrix}
,
\end{align*}
and, consequently,
\begin{align*}
\mathbb{E}[\alpha_X]= %
\begin{pmatrix}1\\
1\\
\end{pmatrix} %
,\quad  C[
\alpha_X]= %
\begin{pmatrix}1&1\\
1&1\\
\end{pmatrix} %
\text{\quad and\quad } \operatorname{Cov}(
\alpha_X)=0.
\end{align*}
Consider now the random symmetric convex set $Y=R_\eta(X)$ where $\eta$
is a uniform random variable on $[0,\pi]$. Then the mean and covariance
of its Feret diameter can be computed \textup{(}see \eqref
{eq:HHIsotropicregZonotope} to \eqref{eq:exppression_k_s}\textup{)}:
\begin{align*}
\mathbb{E}\bigl[H_Y^{(N)}\bigr]= %
\begin{pmatrix}
\frac{4}{\pi}\\[3pt]
\frac{4}{\pi}\\
\end{pmatrix}
\text{\quad and\quad }C\bigl[H_Y^{(N)}\bigr]=\biggl(1+
\frac{2}{\pi}\biggr) %
\begin{pmatrix}1&1\\
1&1\\
\end{pmatrix} %
.
\end{align*}
So
\begin{align*}
\mathbb{E}[\alpha_Y]=\frac{4}{\pi} %
\begin{pmatrix}1\\
1\\
\end{pmatrix}
,\quad  C[\alpha_Y]= \biggl(\frac{\pi+2}{\pi} \biggr)
\begin{pmatrix}1&1\\
1&1\\
\end{pmatrix} %
,
\end{align*}
and
\begin{align*}
\operatorname{Cov}(\alpha_Y)=\frac{\pi^2+2\pi-16}{\pi^2} %
\begin{pmatrix}1&1\\
1&1\\
\end{pmatrix} %
.
\end{align*}
The random set $Y$ is approximated by a random rectangle that has a
varying sides $(\operatorname{Cov}(\alpha_Y)\neq0)$. However, $Y$ have the same
geometrical shape as $X$. This example shows that the $\mathcal
{C}^{(N)}_0$-approximation cannot be used to describe the shape of a
random symmetric convex set for small $N$.
\label{ex:isotropicsquare}
\end{example}

In order to describe the shape of a random symmetric convex set as a
zonotope with a small number of faces, we need to have an approximation
insensitive to the rotations. This leads us to the following approximation.
\subsection{Approximation of a random symmetric convex set by an
isotropic random zonotope}

Previously, we have shown that a random symmetric convex set can be
approximated as a random $0$-regular zonotope. However, we have also
shown that such an approximation can be problematic for small values of
$N$. The aim of this section is to define and characterize an
approximation in $ \mathcal{C}^{(N)}_\infty$ that is invariant up to a
rotation and that can be used for not so large~$N$. For this objective,
we give the approximation for an isotropized set of $X$ instead of $X$.
We will show that a random symmetric convex set can be approximated up
to a rotation by an isotropic random regular zonotope.

Let us note $Y=R_z(X)$ the isotropized set of $X$ with $z$ an
independent uniform variable on $[0,\pi]$.
Let $X_\infty^{(N)}$ be a $\mathcal{C}_\infty^{(N)}$-approximation of
$X$. Then
\begin{align*}
\forall\omega\in\varOmega\text{ a.s.},\quad  X_\infty^{(N)}(
\omega)=R_{\tau
(\omega)}\bigl(\tilde{X}_0^{(N)}(\omega)
\bigr).
\end{align*}
According to the definition of the $\mathcal{C}_\infty
^{(N)}$-approximation, the random set $Y_\infty^{(N)}=R_z(X_\infty
^{(N)})$ is a $\mathcal{C}_\infty^{(N)}$-approximation of $Y$.
Consequently, $Y_\infty^{(N)}=R_{z+\tau}(\tilde{X}_0^{(N)})$. Because
of the independence of $\eta$ and $X$, by the property of addition
modulo~$\pi$ the random variable $\eta=z+\tau$ is a uniform random
variable on $[0,\pi]$ independent of $X$. Then $Y_\infty^{(N)}$ is an
isotropic regular zonotope. We will use such a random regular zonotope
as the approximation of $X$ up to a rotation.
\begin{definition}[$\mathcal{C}_\infty^{(N)}$-isotropic approximation]
Let $X$ be a random symmetric convex set, and $Y=R_z(X)$ its
isotropized set. The isotropic random regular zonotope $Y_\infty
^{(N)}=R_z(X_\infty^{(N)})$ is called the \textit{$\mathcal{C}_\infty
^{(N)}$-isotropic approximation} of $X$ and denoted by $\tilde{X}_\infty^{(N)}$.
\end{definition}
\begin{proposition}[Properties of the $\mathcal{C}_\infty ^{(N)}$-isotropic approximation]
Let $X$ be a random symmetric convex set, and $\tilde{X}_\infty^{(N)}$
be its $\mathcal{C}_\infty^{(N)}$-isotropic approximation.
\begin{enumerate}
\item$\tilde{X}_\infty^{(N)}$ is an isotropic random regular zonotope.
\item$\forall\omega\in\varOmega\text{ a.s.},\exists t(\omega)\in[0,\pi
],\; \forall i=1,\dots, N,\; H_X(t+\theta_i)=H_{\tilde{X}_\infty
^{(N)}}(t+\theta_i)$.
\item$\forall\omega\in\varOmega\text{ a.s.},d_P(X,Y_\infty
^{(N)})\rightarrow0\text{ as } N\rightarrow\infty$.
\item The $\mathcal{C}_\infty^{(N)}$-isotropic approximation is
invariant up to a rotation of $X$.
\item If $X$ is a random regular zonotope, then any face length vector
of $\tilde{X}_\infty^{(N)}$ is a face length vector of $X$.
\end{enumerate}
\end{proposition}

\begin{proof}
$\:$
\begin{enumerate}
\item[1.] It is easy to see that $\tilde{X}_\infty^{(N)}$ is an
isotropized set of $X_\infty^{(N)}$. Consequently, it is an isotropic
random regular zonotope.
\item[2--4.]These properties are direct consequences of Theorem~\ref
{thm:2approxuptorot}.
\item[5.] Suppose that $X$ is a random regular zonotope. Then $X$ and
$\tilde{X}_\infty^{(N)}$ coincide up to a random rotation, and any face
length vector of the one is a face length vector of the other
one.\qedhere
\end{enumerate}
\end{proof}

In order to describe the shape of $X$, the best way would be to
characterize the central face length distribution of $\tilde{X}_\infty
^{(N)}$ from information available on $X$. Unfortunately, there is no
way to compute the characteristics of the random process $H_{\tilde
{X}_\infty^{(N)}}$ from those of $H_{X}$. However, the approximation of
the first- and second-order moments of $H_{\tilde{X}_\infty^{(N)}}$ can
be estimated from those of the Feret diameters of an isotropized set of
$X$ (i.e., $H_{Y}$, where $Y$ is an isotropized set of $X$).

\begin{proposition}[Approximation of the moments of the central
face length distribution] Let $X$ be a symmetric random convex set,
$Y$ its isotropized set, $\tilde{X}_\infty^{(N)}$ the $\mathcal
{C}_\infty^{(N)}$-isotropic approximation of $X$, and $\alpha$ the
central face length vector of $\tilde{X}_\infty^{(N)}$.
\begin{enumerate}
\item An approximation of the first- and second-order moments of $\alpha
$ is given by
\begin{equation}
\hat{\mathbb{E}}[\alpha ]=
\frac{\pi}{2N}\mathbb{E}\bigl[H_Y^{(N)}\bigr],
\end{equation}
\begin{equation}
\hat{V}[\alpha]=\frac{1}{N}K(0)^{-1}V\bigl[H_Y^{(N)}
\bigr].
\end{equation}
Such an approximation is consistent as $N\rightarrow\infty$: $\hat
{\mathbb{E}}[\alpha]-\mathbb{E}[\alpha]\rightarrow0 $ and $\hat
{V}[\alpha]-V[\alpha]\rightarrow0 $ as $N\rightarrow\infty$.
\item If $\hat{\alpha}$ is a positive random vector satisfying $V[\hat
{\alpha}]=\hat{V}[\alpha],\;\mathbb{E}[\hat{\alpha}]=\hat{\mathbb
{E}}[\alpha]$, and $\eta$ an independent uniform variable on $[0,\pi]$,
then the random set $\hat{X}$ defined as
\begin{equation}
\hat{X}=R_\eta\Biggl(\bigoplus_{i=1}^N
\hat{\alpha}_i S_{\theta_i}\Biggr)
\end{equation}
satisfies $\mathbb{E}[U(X)]=\mathbb{E}[U(\hat{X})]$.
\end{enumerate}
\end{proposition}

\begin{proof}
$\:$
\begin{enumerate}
\item[1.] The consistence of the estimate is trivial regarding that $
\mathbb{E}[H_Y^{(N)}]\rightarrow\mathbb{E}[H_{\tilde{X}_\infty
^{(N)}}^{(N)}]$ and $ V[H_Y^{(N)}]\rightarrow V[H_{\tilde{X}_\infty
^{(N)}}^{(N)}]$
as $N\rightarrow\infty$.
\item[2.] Let $\hat{\alpha}$ be a positive random vector satisfying
$V[\hat{\alpha}]=\hat{V}[\alpha],\;\mathbb{E}[\hat{\alpha}]=\hat{\mathbb
{E}}[\alpha]$, and $\eta$ an independent uniform variable on $[0,\pi]$.
Because of the isotropy of $Y$, the vector $ \mathbb{E}[H^{(N)}_Y]$ has
all its components equal to $\frac{1}{\pi}\mathbb{E}[(U(Y)]$, and the
random set $\hat{X}=R_\eta(\bigoplus_{i=1}^N \hat{\alpha}_i S_{\theta
_i})$ satisfies
\begin{align*}
\mathbb{E}\bigl[U(\hat{X})\bigr]&=2\sum_{i=1}^N
\mathbb{E}[\hat{\alpha}_i]
\\
&=\mathbb{E}\bigl[U(Y)\bigr]
\\
&=\mathbb{E}\bigl[U(X)\bigr].\qedhere
\end{align*}
\end{enumerate}
\end{proof}

\begin{remark}
Firstly, note that the quantities $\mathbb{E}[H_Y^{(N)}]$ and
$V[H_Y^{(N)}]$ are easily obtained from the mean and autocovariance of
$H_X$ by using property~\ref{prop:isotroVersionMomentsFeret}.
The approximations $\hat{\mathbb{E}}[\alpha]$ and $\hat{V}[\alpha]$ can
be regarded as the characteristics of the central face length vector of
an isotropic random regular zonotope $\hat{X}$, which has the same
Feret diameter on the $\theta_i$ as an isotropized set of $X$. In
particular, such a zonotope has the same mean perimeter as $X$.

Furthermore, if $X$ is an $N$th random regular zonotope, then such
quantities coincide with those of a face length vector of $X$.
Consequently it is more interesting to use the $\mathcal{C}_\infty
^{(N)}$-isotropic approximation when $X$ is assumed to be an $N$th
random regular zonotope.
\end{remark}

\section{Conclusions and prospects}
In this paper, we proposed different approximations of a symmetric
convex set as a zonotope. These approximations have been further
generalized to random symmetric convex sets. We have shown that a
random convex set can be approximated as precisely as we want as a
random zonotope in terms of the Hausdorff distance.
More specifically, for a random symmetric convex set $X$, the first-
and second-order moments of the face length vector of its zonotope
approximation can be computed from the first- and second-order moments
of the Feret diameter process of $X$.

This work involves several perspectives. The first one would be to get
higher moments of the central face length distribution and to
generalize this work in higher dimension. One potential application of
this work would be to describe the primary grain of the germ--grain
model. Indeed, in a large class of such models, there exist estimators
for the moments of the Feret diameter of the primary grain \cite{GSI}.
In particular, we prospect to apply this to the images of oxalate
ammonium crystals \xch{modeled}{modelled} by the Boolean model (see \citep
{GSI,ICSIA}). However, we need to study the estimators involved by the
zonotope approximation in those germ--grain models.
\bibliographystyle{vmsta-mathphys}
%\bibliography{rahmani_submission}

\end{document}